\documentclass[10pt]{amsart}

\usepackage[colorlinks]{hyperref}
\usepackage{color,graphicx,shortvrb}
\usepackage[latin 1]{inputenc}

\usepackage[active]{srcltx} 

\usepackage{enumerate}
\usepackage{amssymb}
\usepackage{tikz-cd}

\newtheorem{theorem}{Theorem}[section]

\newtheorem{corollary}[theorem]{Corollary}

\newtheorem{definition}[theorem]{Definition}
\newtheorem{lemma}[theorem]{Lemma}

\newtheorem{proposition}[theorem]{Proposition}
\newtheorem{remark}[theorem]{Remark}

\def\J#1#2#3{ \left\{ #1,#2,#3 \right\} }

\def\NN{{\mathbb{N}}}
\def\11{\textbf{$1$}}
\def\CC{{\mathbb{C}}}

\DeclareGraphicsExtensions{.jpg,.pdf,.png,.eps}

\begin{document}

\title[On the extension of isometries between the unit spheres]{On the extension of isometries between the unit spheres of a JBW$^*$-triple and a Banach space}

\author[Becerra-Guerrero, Cueto-Avellaneda, Fern\'{a}ndez-Polo, Peralta]{Julio Becerra-Guerrero, Mar{\'i}a Cueto-Avellaneda, Francisco J. Fern\'{a}ndez-Polo, Antonio M. Peralta}

\address[J. Becerra Guerrero, M. Cueto-Avellaneda, F.J. Fern\'{a}ndez-Polo, A.M. Peralta]{Departamento de An{\'a}lisis Matem{\'a}tico, Facultad de
Ciencias, Universidad de Granada, 18071 Granada, Spain.}
\email{juliobg@ugr.es, mcueto@ugr.es, pacopolo@ugr.es, aperalta@ugr.es}


\subjclass[2010]{Primary 47B49, Secondary 46A22, 46B20, 46B04, 46A16, 46E40.}

\keywords{Tingley's problem; Mazur--Ulam property; extension of isometries; JBW$^*$-triples}

\date{}

\begin{abstract} We prove that every JBW$^*$-triple $M$ with rank one or rank bigger than or equal to three satisfies the Mazur--Ulam property, that is, every surjective isometry from the unit sphere of $M$ onto the unit sphere of another Banach space $Y$ extends to a surjective real linear isometry from $M$ onto $Y$. We also show that the same conclusion holds if $M$ is not a JBW$^*$-triple factor, or more generally, if the atomic part of $M^{**}$ is not a rank two Cartan factor.
\end{abstract}

\maketitle
\thispagestyle{empty}

\section{Introduction}

Inspired by the Mazur--Ulam theorem and the positive answers obtained to Tingley's problem in a wide range of concrete spaces, L. Cheng and Y. Dong introduced in \cite{ChenDong2011} the Mazur--Ulam property. A Banach space $X$ satisfies the \emph{Mazur--Ulam property} if for any Banach space $Y$, every surjective isometry $\Delta: S(X)\to S(Y)$ admits an extension to a surjective real linear isometry from $X$ onto $Y$, where $S(X)$ and $S(Y)$ denote the unit spheres of $X$ and $Y$, respectively.\smallskip

The so-called \emph{Tingley's problem} asks if every surjective isometry between the unit spheres of two Banach spaces $X$ and $Y$ admits an extension to a surjective real linear isometry between the spaces. This problem on Banach spaces was first considered by D. Tingley in \cite{Ting1987}. Recent positive solutions to Tingley's problem in concrete settings include surjective isometries $\Delta: S(X)\to S(Y)$ when $X$ and $Y$ are von Neumann algebras \cite{FerPe17d}, compact C$^*$-algebras and JB$^*$-triples \cite{PeTan16} and \cite{FerPe18Adv}, atomic JBW$^*$-triples \cite{FerPe17c}, spaces of trace class operators \cite{FerGarPeVill17}, spaces of $p$-Schatten von Neumann operators with $1\leq p\leq \infty$ \cite{FerJorPer2018}, preduals of von Neumann algebras and the self-adjoint parts of two von Neumann algebras \cite{Mori2017}. 
The reader is referred to the surveys \cite{Ding2009,YangZhao2014,Pe2018} for a more thorough overview on Tingley's problem.\smallskip

The available literature shows that some of the spaces for which Tingley's problem admits a positive solution actually satisfy the stronger Mazur--Ulam property. That is the case of $c_0(\Gamma,\mathbb{R})$, $\ell_{\infty}(\Gamma,\mathbb{R})$ (see \cite[Corollary 2]{Ding07}, \cite[Main Theorem]{Liu2007}), $C(K,\mathbb{R})$ where $K$ is a compact Hausdorff space \cite[Corollary 6]{Liu2007}, $L^{p}((\Omega, \mu), \mathbb{R})$ where  $(\Omega, \mu)$ is a $\sigma$-finite measure space and $1\leq p\leq \infty$ \cite{Ta:1,Ta:8,Ta:p}, almost-CL-spaces admitting a smooth point \cite{THL,TL}, $c_0(\Gamma)= c_0(\Gamma,\mathbb{C})$ \cite{JVMorPeRa2017}, $\ell_{\infty}(\Gamma,\mathbb{C})$ \cite{Pe2017}, and commutative von Neumann algebras \cite{CuePer}. The list has been widen in a very recent result by M. Mori and N. Ozawa in \cite{MoriOza2018} where they prove that every unital complex C$^*$-algebra and every real von Neumann algebra satisfies the Mazur--Ulam property.\smallskip

Our main goal in this note is to establish a version of the results by M. Mori and N. Ozawa in the setting of JBW$^*$-triples (see section \ref{sec:2} for concrete definitions and examples). In our principal result (see Theorem \ref{t JBW-triples satisfy the MUP} and Proposition \ref{p rank one JBW-triples}) we prove that every JBW$^*$-triple with rank one or rank bigger than or equal to three satisfies the Mazur--Ulam property. We actually prove a stronger result by showing that if $M$ is a JBW$^*$-triple such that the atomic part of $M^{**}$ is not a Cartan factor of rank two (in particular when $M$ is not a factor), then $M$ satisfies the Mazur--Ulam property (see Remark \ref{remark final}).\smallskip

The starting point in our arguments is Corollary \ref{p complex JBW*-triples have the strong Mankiewicz property} where we check, by applying a result  due to M. Mori and N. Ozawa \cite{MoriOza2018}, that the closed unit ball of a JBW$^*$-triple satisfies the strong Mankiewicz property.\smallskip

In section \ref{sec:3} we deepen our knowledge on a class of faces of the closed unit ball of the bidual of a JB$^*$-triple which remained unexplored until now. The main result in \cite{EdFerHosPe2010} shows that the proper norm closed faces of the closed unit ball, $\mathcal{B}_E$, of a JB$^*$-triple, $E$, are in one-to-one correspondence with those compact tripotents in $E^{**}$. A preceding result due to C.M. Edwards and G.T. R\"{u}ttimann assures that weak$^*$ closed proper faces of the closed unit ball of $E^{**}$ are in one-to-one correspondence with the set of tripotents in $E^{**}$ (see \cite{EdRutt88}). On the other hand, following the notation in \cite[\S 2]{FerPe06}, we shall say that a set $S\subseteq E^{**}$ is \emph{open relative to $E$} if $S\cap E$ is $\sigma(E^{**},E^*)$ dense in $\overline{S}^{\sigma(E^{**},E^*)}.$ It seems natural to ask whether relatively open faces in the closed unit ball of $E^{**}$ can be characterized in terms of a set of tripotents in $E^{**}$. We shall show in Theorem \ref{t characterization relatively open weak* closed faces} that a proper weak$^*$ closed face of the closed unit ball of $E^{**}$ is open relative to $E$ if and only if it is a weak$^*$ closed face associated with a compact tripotent in $E^{**}$.\smallskip

Let $E$ be a JB$^*$-triple. The characterization of those proper weak$^*$ closed faces of the closed unit ball of $E^{**}$ which are open relative to $E$ in terms of the compact tripotents in $E^{**}$ is applied to establish that if $\Delta : S(M)\to S(Y)$ is a surjective isometry, where $M$ is a JBW$^*$-triple and $Y$ is a Banach space, then the restriction of $\Delta$ to each norm closed proper face of $\mathcal{B}_{M}$ is an affine mapping (see Proposition \ref{p affinity on closed faces}).

\section{Background on JB$^*$-triples and the strong Mankiewicz property}\label{sec:2}

Along this paper, given a complex Banach space $X,$ its underlying real Banach space will be denoted by the same symbol $X$ or by $X_{\mathbb{R}}$ in case of ambiguity. It is well known that $\varphi\mapsto \Re\hbox{e}\varphi$ is an isometric bijection from $(X^*)_{\mathbb{R}}$ onto $(X_{\mathbb{R}})^*.$ If $X$ is a real or complex Banach space, the symbol $\mathcal{B}_{X}$ will stand for the closed unit ball of $X$, while $S(X)$ will denote the unit sphere of $X$. We shall frequently regard $X$ as being contained in
$X^{**}$ and we identify the weak$^*$-closure in $X^{**}$ of a closed subspace $Y$ of $X$ with $Y^{**}$.\smallskip

A convex subset $K$ of a normed space $X$ is called a \emph{convex body} if it has non-empty interior in $X$. The Mazur--Ulam theorem was extended by P. Mankiewicz in \cite{Mank1972} by showing that every surjective isometry between convex bodies in two arbitrary normed spaces can be uniquely extended to an affine function between the spaces. This result is one of the main tools applied in those papers devoted to explore new progress to Tingley's problem and to determine new Banach spaces satisfying the Mazur--Ulam property.\smallskip

In a very recent paper by M. Mori and N. Ozawa (see \cite{MoriOza2018}), a new technical achievement has burst into the scene of the current research on those Banach spaces satisfying the Mazur--Ulam property. Following these authors, we shall say that a convex subset $K$ of a normed space $X$ satisfies the \emph{strong Mankiewicz property} if every surjective isometry $\Delta$ from $K$ onto an arbitrary convex subset $L$ in a normed space $Y$ is affine. As observed by Mori and Ozawa, every convex subset of a strictly convex normed space satisfies the strong Mankiewicz property because it is uniquely geodesic (see \cite[Lemma 6.1]{BadFurGeMo2007}), and there exist examples of convex subsets of $L^1[0, 1]$ which do not satisfy this property (see \cite[Example 5]{MoriOza2018}). 
In \cite[Theorem 2]{MoriOza2018} Mori and Ozawa show that some of the hypotheses in Mankiewicz's theorem can be somehow relaxed. The following result has been borrowed from \cite[Theorem 2 and its proof]{MoriOza2018}.

\begin{theorem}\label{t Mori--Ozawa strong Mankiewicz}\cite[Theorem 2]{MoriOza2018} Let $X$ be a Banach space such that the closed convex hull of the extreme points, $\partial_e (\mathcal{B}_X),$ of the closed unit ball, $\mathcal{B}_X$, of $X$ has non-empty interior in $X$. Then every
convex body $K\subset X$ satisfies the strong Mankiewicz property. Furthermore, suppose $L$ is a convex subset of a normed space $Y$, and $\Delta : \mathcal{B}_{X}\to L$ is a surjective isometry. Then $\Delta$ can be uniquely extended to an affine isometry from $X$ onto a norm closed subspace of $Y$.$\hfill\Box$
\end{theorem}

The celebrated Russo--Dye theorem (see \cite{RuDye}) assures that every (complex) unital C$^*$-algebra satisfies the hypotheses in the previous theorem. Actually, Mori and Ozawa show that any Banach space in the class of real von Neumann algebras also satisfies the desired hypotheses (see \cite[Corollary 3]{MoriOza2018}). This can be also deduced from the real version of the Russo--Dye theorem, established by J.C. Navarro and M.A. Navarro in \cite[Corollary 6]{NavNav2012}, which asserts that the open unit ball of a real von Neumann algebra $A$ is contained in the sequentially convex hull of the set of unitary elements in $A$. As pointed out by Mori and Ozawa in \cite[Proof of Corollary 3]{MoriOza2018}, the latter conclusion can be deduced from a result due to B. Li (see \cite[Theorem 7.2.4]{Li70}).\smallskip

Let us continue this section by adding some new examples of Banach spaces fulfilling the hypotheses in the Mori--Ozawa Theorem \ref{t Mori--Ozawa strong Mankiewicz}. 
Henceforth, let $B(H,K)$ denote the Banach space of all bounded linear operators between two complex Hilbert spaces $H$ and $K$. A J$^*$-algebra in the sense introduced by L.A. Harris in \cite{Harris74} is a closed complex subspace $E$ of $B(H,K)$ such that $a a^* a\in E$ whenever $a\in E$. Harris proved in \cite[Corollary 2]{Harris74} that the open unit ball, $\stackrel{\circ}{\mathcal{B}}_{E},$ of every J$^*$-algebra $E$ is a bounded symmetric domain (i.e. for each $x\in \stackrel{\circ}{\mathcal{B}}_{E}$ there exists a biholomorphic mapping in Fr\'{e}chet's sense $h: \stackrel{\circ}{\mathcal{B}}_{E}\to \stackrel{\circ}{\mathcal{B}}_{E}$ such that $h$ has $x$ as its only fixed point and $h^2 $ is the identity map on $\stackrel{\circ}{\mathcal{B}}_{E}$). However, J$^*$-algebras are not the unique complex Banach spaces whose open unit ball is a bounded symmetric domain. W. Kaup established in \cite{Ka83} that the open unit ball of a complex Banach space $E$ is a bounded symmetric domain if and only if $E$ is a JB$^*$-triple, that is, there exists a continuous triple product $\J ... :
E\times E\times E \to E,$ which is symmetric and linear in the
first and third variables, conjugate linear in the second variable,
and satisfies the following axioms:
\begin{enumerate}[{\rm (a)}] \item (Jordan identity) $L(a,b) L(x,y) = L(x,y) L(a,b) + L(L(a,b)x,y)
 - L(x,L(b,a)y),$ for every $a,b,x,y$ in $E$, where $L(a,b)$ is the operator on $E$ given by $L(a,b) x = \J abx;$
\item $L(a,a)$ is a hermitian operator with non-negative spectrum for all $a\in E$;
\item $\|\{a,a,a\}\| = \|a\|^3$ for each $a\in E$.\end{enumerate}

Every J$^*$-algebra is a JB$^*$-triple with respect to the triple product given by
\begin{equation}\label{eq product operators} \J xyz =\frac12 (x y^* z +z y^*
x).\end{equation} Consequently, C$^*$-algebras and complex Hilbert spaces are JB$^*$-triples with respect to the above triple product. Another interesting examples are given by Jordan structures; for example every JB$^*$-algebra in the sense considered in \cite{Wri77,WriYou77} and \cite{Sidd2007,Sidd2010} are JB$^*$-triples under the triple
product \begin{equation}\label{eq triple product JBstaralg} \J xyz = (x\circ y^*) \circ z + (z\circ y^*)\circ x -
(x\circ z)\circ y^*.
\end{equation}

Another milestone result in the theory of JB$^*$-triples is the so-called \emph{Kaup-Banach-Stone theorem}, established by W. Kaup in \cite[Proposition 5.5]{Ka83}, which proves that a linear bijection between JB$^*$-triples is an isometry if and only if it is a triple isomorphism.\smallskip

A JBW$^*$-triple is a JB$^*$-triple which is also a dual Banach space (with a unique isometric predual \cite{BarTi86}). It is known
that the second dual of a JB$^*$-triple is a JBW$^*$-triple (compare \cite{Di86}). An extension of Sakai's theorem assures that the triple product of every JBW$^*$-triple is separately weak$^*$ continuous (cf. \cite{BarTi86} or \cite{Horn87}).\smallskip

We shall only recall some basic facts and results in the theory of JB$^*$-triples. Let $A$ be a C$^*$-algebra regarded as a JB$^*$-triple with the product given in \eqref{eq product operators}. It is easy to see that partial isometries in $A$ are precisely those elements $e$ in $A$ such that $\{e,e,e\}=e$. An element $e$ in a JB$^*$-triple $E$ is said to be a \emph{tripotent} if $\{e,e,e\} =e$. The extreme points of the closed unit ball of a JB$^*$-triple can only be understood in terms of those tripotents satisfying an additional property. For each tripotent $e$ in $E$ the eigenvalues of the operator $L(e,e)$ are contained in the set $\{0,1/2,1\},$ and $E$ can be decomposed in the form  $$E= E_{2} (e) \oplus E_{1} (e) \oplus E_0 (e),$$ where for
$i=0,1,2,$ $E_i (e)$ is the $\frac{i}{2}$ eigenspace of $L(e,e)$. This decomposition is known as the \emph{Peirce decomposition}\label{eq Peirce decomposition} associated with $e$. The so-called \emph{Peirce arithmetic} affirms that for every $i,j,k\in\{0,1,2\}$ we have \begin{enumerate}[$\bullet$]
\item $\J {E_{i}(e)}{E_{j} (e)}{E_{k} (e)}\subseteq E_{i-j+k} (e)$ if $i-j+k$ belongs to the set  $\{ 0,1,2\},$
and $\J {E_{i}(e)}{E_{j} (e)}{E_{k} (e)}=\{0\}$ otherwise;
\item $\{ E_{2} (e), E_{0}(e), E\} = \{ E_{0} (e), E_{2}(e), E\} =0.$
\end{enumerate}
For $k\in \{0,1,2\}$, the projection $P_{k_{}}(e)$ of $E$ onto $E_{k} (e)$ is called the Peirce $k$-projection. It is known that Peirce projections are contractive (cf. \cite{FriRu85}) and satisfy that $P_{2}(e) = Q(e)^2,$ $P_{1}(e) =2(L(e,e)-Q(e)^2),$ and $P_{0}(e) =Id_E - 2 L(e,e) + Q(e)^2,$ where for each $a\in E$, $Q(a):E\to E$ is the conjugate linear map given by $Q(a) (x) =\{a,x,a\}$.
A tripotent $e$ in $E$ is called \emph{unitary} (respectively, \emph{complete} or \emph{maximal}) if $E_2 (e) = E$ (respectively, $E_0 (e) =\{0\}$). Finally, a tripotent $e$ in $E$ is said to be \emph{minimal} if $E_2(e)=\CC e \neq \{0\}$. 
\smallskip

Additional properties of the Peirce decomposition assure that the Peirce space $E_2 (e)$ is a unital JB$^*$-algebra with unit $e$,
product $x\circ_e y := \J xey$ and involution $x^{*_e} := \J exe$, respectively. It follows from Kaup-Banach-Stone theorem that the triple product in $E_2 (e)$ is uniquely determined by the identity $$\{ a,b,c\} = (a \circ_e b^{*_e}) \circ_e c + (c\circ_e b^{*_e}) \circ_e a - (a\circ_e c) \circ_e b^{*_e}, \ \ \ \ (\forall a,b,c\in E_2 (e)).$$ Furthermore, for each $x\in E$ the element \begin{equation}\label{eq positive element FR}\hbox{$P_2(e) \{x,x,e\} =  \{P_2(e) (x),P_2(e)(x),e\} + \{P_1(e) (x),P_1(e)(x),e\}$}
 \end{equation}  is positive in $E_2(e)$, and $P_2(e) \{x,x,e\} =0$ if and only if $P_j(e) (x) =0$ for every $j=1,2$ (see \cite[Lemma 1.5 and preceding comments]{FriRu85}
).\smallskip

Elements $a,b$ in a JB$^*$-triple $E$ are called \emph{orthogonal} (written $a\perp b$) if $L(a,b) =0$. It is known that $a\perp b$ $\Leftrightarrow$ $\J aab =0$ $\Leftrightarrow$ $\{b,b,a\}=0$ $\Leftrightarrow$ $b\perp a;$ (see, for example, \cite[Lemma 1]{BurFerGarMarPe}). Let $e$ be a tripotent in $E$. It follows from the Peirce arithmetic that $a\perp b$ for every $a\in E_2(e)$ and every $b\in E_0(e)$.
\smallskip

The rank of a JB$^*$-triple $E$ is the minimal cardinal number $r$ satisfying $\hbox{card}(S)\leq r$ whenever $S$ is an orthogonal subset of
$E$, that is, $0\notin S$ and $x\perp y$ for every $x\neq y$ in $S$. \smallskip

We shall consider the following partial order on the set of tripotents of a JB$^*$-triple $E$ defined by  $u \leq e$ if $e-u$ is a tripotent in $E$ and $e-u \perp u$. It is known that $u\leq e$ if and only if $u$ is a projection in the JB$^*$-algebra $E_2(e)$.\smallskip

Similarly as there exist C$^*$-algebras containing no non-zero projections, we can find JB$^*$-triples containing no non-trivial tripotents. Another geometric property of JB$^*$-triples provides an algebraic characterization of the extreme points of their closed unit balls. Concretely, the (complex and the real) extreme points of the closed unit ball of a JB$^*$-triple $E$ are precisely the complete tripotents in $E$, that is \begin{equation}\label{eq extreme points and complete tripotents} \partial_e(\mathcal{B}_{E}) =\{ \hbox{complete tripotents in } E \},
\end{equation} (cf. \cite[Lemma 4.1]{BraKaUp78} and \cite[Proposition 3.5]{KaUp77}).\smallskip

An element $u$ in a unital C$^*$-algebra $A$ is called \emph{unitary} if $u u^* = u^* u ={1}$.  It is known that an element $u$ in a JB$^*$-algebra $B$ is a unitary tripotent if and only if $u$ is Jordan invertible in $B$ and its unique Jordan inverse in $B$ coincides with $u^*$ (compare \cite{WriYou77} and \cite{Sidd2007,Sidd2010}). If a unital C$^*$-algebra $A$ is regarded as a JB$^*$-algebra with the natural Jordan product given by $a\circ b := \frac12 (a b + ba)$, then an element $u$ in $A$ is a unitary in the C$^*$-algebra sense if, and only if, it is unitary in the JB$^*$-algebra sense if, and only if, it is unitary (tripotent) in the JB$^*$-triple sense. Clearly, every unitary element in a JB$^*$-algebra is an extreme point of its closed unit ball.\smallskip

After reviewing the basic facts on the extreme points of the closed unit ball of a JB$^*$-triple, we can next consider the strong Mankiewicz property for convex bodies in a JBW$^*$-triple. Let us recall that the Russo--Dye theorem is the tool employed by Mori and Ozawa to show, via Theorem \ref{t Mori--Ozawa strong Mankiewicz} \cite[Theorem 2]{MoriOza2018}, that every convex body of a unital C$^*$-algebra satisfies the strong Mankiewicz property. The Russo--Dye theorem was extended to the setting of unital JB$^*$-algebras by J.D.M. Wright and M.A. Youngson \cite{WriYou77} and A.A. Siddiqui \cite{Sidd2010}. 
In 2007, A.A. Siddiqui proved that every element in the unit ball of a JBW$^*$-triple is the average of two extreme points (see \cite[Theorem 5]{Sidd2007}). Our next result is a straight consequence of this result and \cite[Theorem 2]{MoriOza2018}.

\begin{corollary}\label{p complex JBW*-triples have the strong Mankiewicz property} The closed unit ball of every JBW$^*$-triple $M$ satisfies the strong Mankiewicz property. Consequently, every convex body in a JBW$^*$-triple satisfies the same property. Furthermore, if $L$ is a convex subset of a normed space $Y$, then every surjective isometry $\Delta : \mathcal{B}_{M}\to L$ can be uniquely extended to an affine isometry from $M$ onto a norm closed subspace of $Y$.
\end{corollary}

\section{Relatively open faces in the bidual of a JB$^*$-triple}\label{sec:3}

As in the study of the Mazur--Ulam property in the setting of unital C$^*$- and von Neumann algebras (see \cite{MoriOza2018}), the facial structure of JB$^*$-triples plays a central role in our study of the Mazur--Ulam property in the spaces belonging to the class of JBW$^*$-triples. For this purpose we shall require some basic notions.\smallskip

In order to understand the nomenclature we refresh the usual ``facear'' and ``pre-facear'' operations.
Let $X$ be a complex Banach space with dual space $X^*$. For each subset $F\subseteq \mathcal{B}_{X}$ and each $G\subseteq \mathcal{B}_{X^*}$, let
\begin{equation}\label{1.1}
F^{\prime} = \{a \in \mathcal{B}_{X^*} : a(x) = 1\,\, \forall x \in F\},\quad G_{\prime} = \{x \in \mathcal{B}_{X} :a(x) = 1\,\, \forall a \in G\}.
\end{equation}
Then, $F^{\prime}$ is a weak$^*$ closed face of $\mathcal{B}_{X^*}$ and
$G_{\prime}$ is a norm closed face of $\mathcal{B}_{X}$. The subset $F$ of $\mathcal{B}_{X}$ is said to be
a \emph{norm-semi-exposed face} of $\mathcal{B}_{X}$ if $F=(F^{\prime})_{\prime}$
and the subset $G$ of $\mathcal{B}_{X^*}$ is said to be a \emph{weak$^*$-semi-exposed face} of $\mathcal{B}_{X^*}$ if
$G=(G_{\prime})^{\prime}$. The mappings $F \mapsto F^{\prime}$
and $G \mapsto G_{\prime}$ are anti-order isomorphisms between
the complete lattices 
of norm-semi-exposed faces of $\mathcal{B}_{X}$ and 
of weak$^*$-semi-exposed faces of $\mathcal{B}_{X^*}$ and are inverses of each other.\smallskip

In a celebrated result published in \cite{EdRutt88}, C.M. Edwards and G.T. R\"{u}ttimann proved that the weak$^*$ closed faces of the closed unit ball of a JBW$^*$-triple $M$ are in one-to-one correspondence with the tripotents in $M$. The concrete theorem reads as follows:

\begin{theorem}\label{t ER weakstar closed faces}\cite{EdRutt88} Let $M$ be a JBW$^*$-triple, and let $F$ be a weak$^*$ closed face of
the unit ball $\mathcal{B}_M$ in $M$. Then, there exists a tripotent
$e$ in $M$ such that
$$F =F^{M}_e= e + \mathcal{B}_{_{M_0(e)}} = \left(\{e\}_{\prime}\right)^{\prime},$$
where $\mathcal{B}_{_{M_0(e)}}$ denotes the unit ball of the Peirce zero space $M_0(e)$ in $M$. Furthermore,
the mapping $e \mapsto F^M_e= \left(\{e\}_{\prime}\right)^{\prime}$ is an anti-order isomorphism from the partially ordered set
${\mathcal{U}} (M)$ of all tripotents in $M$ onto the partially ordered set 
of weak$^*$ closed faces of $\mathcal{B}_M$ excluding the empty set.$\hfill\Box$
\end{theorem}

We continue by reviewing the notion of compact tripotent in the second dual of a JB$^*$-triple. Given an element $a$ in a JB$^*$-triple, we set $a^{[1]} := a$, $a^{[3]} := \{a,a,a\}$, and $a^{[2n+1]} :=
\{a,a,a^{[2n-1]}\},$ $(n\in \NN)$. Let us fix a JBW$^*$-triple $M$. It is known that, for each $a\in S(M)$, the sequence $(a^{[2n -1]})$ converges in the weak$^*$ topology of $M$ to a (possibly zero) tripotent $u_{_M}(a)$ or $u(a)$ in $M$ (compare \cite[Lemma 3.3]{EdRutt88} or \cite[page 130]{EdFerHosPe2010}). This tripotent $u_{_{M}}(a)$ is called the \emph{support} \emph{tripotent} of $a$. 
The equality $a=u(a) + P_0 (u(a)) (a)$ holds for every $a$ in the above conditions. For a norm-one element $a$ in a JB$^*$-triple $E$, $u_{_{E^{**}}}(a)$ will denote the support tripotent of $a$ in $E^{**}$ which is always non-zero. Given $a$ in $M$ the support tripotents $u_{_M}(a)$ and $u_{_{M^{**}}}(a)$ need not coincide. To avoid confusion, given a norm-one element $a$ in a JBW$^*$-triple $M$, unless otherwise stated, we shall write $u(a)$ for the support tripotent of $a$ in $M^{**}$. \smallskip

Accordingly to the terminology introduced by C.M. Edwards and G.T. R\"{u}ttimann in \cite{EdRu96}, a tripotent $e$ in the second dual, $E^{**}$, of a JB$^*$-triple $E$ is said to be \emph{compact-$G_{\delta}$} if there exists a norm-one element $a$ in $E$ satisfying $u(a)=u_{_{E^{**}}}(a) =e$. A tripotent $e$ in $E^{**}$ is \emph{compact} if $e=0$ or it is the infimum of a decreasing net of compact-$G_{\delta}$ tripotents in $E^{**}$ converging to $e$ in the weak$^*$ topology. 
Clearly, every tripotent in $E$ is compact in $E^{**}$.\smallskip

C.A. Akemann and G.K. Pedersen described in \cite{AkPed92} the facial structure of a general C$^*$-algebra, a task actually initiated and considered by C.M. Edwards and G.T. R\"{u}ttimann in \cite{EdRutt88}. The understanding of the facial structure of a general JB$^*$-triple was completed by C.M. Edwards, F.J. Fern{\'a}ndez-Polo, C.S. Hoskin and A.M. Peralta in \cite{EdFerHosPe2010}. The result required in this note is subsumed in the next theorem.

\begin{theorem}\label{thm norm closed faces}\cite[Theorem 3.10 and Corollary 3.12]{EdFerHosPe2010} Let $E$ be a JB$^*$-triple, and let $F$ be a norm closed face of
the unit ball $\mathcal{B}_E$ in $E$. Then, there exists a {\rm(}unique{\rm)} compact tripotent
$u$ in $E^{**}$ such that
$$F =F^{E}_u= (u + \mathcal{B}_{_{E_0^{**}(u)}}) \cap E = \left({\{u\}}_{\prime}\right)_{\prime},$$
where $\mathcal{B}_{_{E_0^{**}(u)}}$ denotes the unit ball of the Peirce zero space $E^{**}_0(u)$
in $E^{**}$. Furthermore,
the mapping $u \mapsto F_u^{E}= \left({\{u\}}_{\prime}\right)_{\prime}$ is an anti-order isomorphism from the partially ordered set
of all compact tripotents in $E^{**}$ onto the partially ordered set 
of norm closed faces of $\mathcal{B}_E$ excluding the empty set.$\hfill\Box$
\end{theorem}

The facial structure of the closed unit ball of a JB$^*$-triple $E$ assures that norm closed faces of $\mathcal{B}_{E}$ are in one-to-one correspondence with compact tripotents in $E^{**}$. Even in the case in which we are dealing with a JBW$^*$-triple $M$, tripotents in $M$ are not enough to determine all norm closed faces of $\mathcal{B}_{M}$.\smallskip

The celebrated Kadison's transitivity theorem was extended by L.J. Bunce, J. Mart{\'i}nez-Moreno and the last two authors of this note to the setting of JB$^*$-triples (cf. \cite[Theorem 3.3]{BuFerMarPe}). Suppose $E$ is a JB$^*$-triple. A consequence of Kadison's transitivity theorem proves that every maximal norm closed proper face of $\mathcal{B}_E$ is of the form \begin{equation}\label{eq maximal proper faces} F_e^{E}= (e + \mathcal{B}_{_{E_0^{**}(e)}}) \cap E,
\end{equation} where $e$ is a minimal tripotent in $E^{**}$ (see \cite[Corollary 3.5]{BuFerMarPe}).\smallskip

When comparing Theorems \ref{t ER weakstar closed faces} and \ref{thm norm closed faces} the natural question is whether we can topologically distinguish between weak$^*$ closed faces in $\mathcal{B}_{E^{**}}$ associated with compact tripotents in $E^{**}$ from weak$^*$ closed faces in $\mathcal{B}_{E^{**}}$ associated with non-compact tripotents in $E^{**}$. We shall see in Theorem \ref{t characterization relatively open weak* closed faces} that the required topological notion was already considered in \cite{FerPe06}.\smallskip

Let $X$ be a Banach space, $E$ a weak$^*$ dense subset of $X^*$ and $S$ a non-zero subset of $X^{*}$. Following the notation in \cite[\S 2]{FerPe06}, we shall say that $S$ is \emph{open relative to $E$} if $S\cap E$ is $\sigma(X^{*},X)$ dense in $\overline{S}^{\sigma(X^{*},X)}.$ Let $E$ be a JB$^*$-triple. A tripotent $e$ in $E^{**}$ is called \emph{closed (relative to $E$)} if $E_{0}^{**} (e)$ is an open subset of $E^{**}$ relative to $E$. We shall say that $e$ is \emph{bounded {\rm(}relative to $E${\rm)}} if there exists $x$ in the unit sphere of $E$ satisfying that $\{e,e,x\} = e$ (or equivalently, $x= e + P_0(e) (x)$ in $E^{**}$). One of the main achievements in \cite{FerPe06} shows that a tripotent $u$ in $E^{**}$ is compact if and only if it is closed and bounded (cf. \cite[Theorem 2.6.]{FerPe06}).\smallskip

Another tools needed for our purposes are the triple functional calculus at an element in a JB$^*$-triple $E$
and the strong$^*$ topology. The symbol $E_a$ will stand for the
JB$^*$-subtriple of $E$ generated by the element $a$\label{eq subtriple single generated}. It is known that
$E_a$ is JB$^*$-triple isomorphic (and hence isometric) to $C_0
(\Omega_{a})$ for some locally compact Hausdorff space
$\Omega_{a}$ contained in $(0,\|a\|],$ such that $\Omega_{a}\cup
\{0\}$ is compact, where $C_0 (\Omega_{a})$ denotes the Banach
space of all complex-valued continuous functions vanishing at $0.$
It is also known that the triple identification of $E_a$ and $C_0
(\Omega_{a})$ can be assumed to satisfy that $a$ correspond to the function
mapping each $\lambda\in \Omega_{a}$ to itself (cf.
\cite[Corollary 1.15]{Ka83} and \cite{FriRu85}). The \emph{triple functional calculus} at the element $a$ is defined as follows.
Given a function $f\in C_{0}(\Omega_{a})$, $f_t (a)$ will stand for the (unique) element
in $E_a$ corresponding to the function $f$. 
\smallskip

Let $a$ be an element in a JB$^*$-triple $E$. Let $g_t(a)=:a^{[\frac12]}\in E_a$ where $g (\lambda) = \lambda^{\frac12}$ ($\lambda\in \Omega_{a}$).  Accordingly to the notation in \cite{FerPe06}, along this paper, $P_0 (a)$ will denote the bounded linear operator on $E$ defined by
\begin{equation}\label{eq Bergamn operator associated with an element} P_0 (a) (y) = y -2 L(a^{[\frac12]},a^{[\frac12]}) (y) +Q(a^{[\frac12]})^2 (y).
\end{equation}  In the literature this operator is called the
Bergman operator associated with $a$. We should note that this
notation is not ambiguous when $a=e$ is a tripotent, because $P_0 (e)$ is precisely the Peirce projection of $E$ onto $E_0
(e).$ In the sequel we shall also write $a^{[2]}$ for the element $h_t(a)\in E_a$ where $h (\lambda) = \lambda^{2}$ ($\lambda\in \Omega_{a}$). The elements $a^{[2]}$ and $a^{[\frac12]}$ may be seen as artificial constructions in the triple setting, however, both of them lie in $E_a$.\smallskip

Let $a$ be a norm-one element in a JBW$^*$-triple $M$. Lemma 3.3 in \cite{EdRu96} implies the existence of a smallest tripotent $r(a)$ in $M$ such that $a\in M_2(r(a))$ and $a$ is positive in the latter JBW$^*$-algebra. Furthermore, in the order of the JBW$^*$-algebra $M_2^{**}(r(a))$, we have $$0 \leq u_{_M} (a)\leq u_{_{M^{**}}} (a) \leq a^{[2n+1]} \leq  a \leq r(a),$$ for every natural $n$. The tripotent $r(a)=r_{_M}(a)$ is called the \emph{range
tripotent} of $a$. We have already commented that the support tripotent $u_{_M}(a)$ might be zero, however, $u_{_{M^{**}}} (a)\neq 0$.\smallskip

It is time to recall the definition and basic properties of the strong$^*$ topology.
Suppose $\varphi$ is a norm-one normal functional in the predual $M_*$ of a JBW$^*$-triple $M.$
If $z$ is a norm-one element in $M$ satisfying $\varphi (z) =1$, then the assignment
$$(x,y)\mapsto \varphi\J xyz$$ defines a positive sesquilinear
form on $M,$ which does not depend on the choice of $z$. We therefore have a prehilbert seminorm on $M$ defined by $\|x\|^2_{\varphi}:= \varphi\J xxz.$ The \emph{strong$^*$ topology} of $M$ 
is the topology on $M$ induced by the seminorms $\|x\|_{\varphi}$ when $\varphi$ ranges in the unit sphere of $M_*$. The strong$^*$ topology was originally introduced in \cite{BarFri}, and subsequently developed in \cite{PeRo} (see also \cite[\S 5.10.2]{CaRo2018v2}). Among the properties of this topology we note that the strong$^*$ topology of $M$ is compatible with the duality $(M,M_{*})$ (see \cite[Theorem 3.2]{BarFri}). By combining this property with the bipolar theorem, we deduce that the identity \begin{equation}\label{eq bipolar
compatible} \overline{C}^{\sigma(M,M_*)} =
\overline{C}^{\tiny\hbox{strong$^*$}},
\end{equation} holds for every convex subset $C\subseteq M.$ Another interesting property asserts that the triple product of $M$ is jointly strong$^*$ continuous on bounded sets of $M$ (see \cite{PeRo} or \cite[Theorem 5.10.133]{CaRo2018v2}).\smallskip

We shall study next a series of geometric inequalities in different settings. The first case is probably part of the folklore in the theory of C$^*$-algebras.

\begin{lemma}\label{l norm inequality Cstr} Let $A$ be a unital C$^*$-algebra. Suppose $a$ and $b$ are two elements in the closed unit ball of $A$ with $a$ positive. Then $\| 1 - a(1 + b)a\|\leq  1$.
\end{lemma}

\begin{proof} Let $z = 1 - a(1 + b)a$. Since for each $y\in A$, the mapping $x\mapsto yxy^*$ is positive, we get
$$zz^* = (1 - a(1 + b)a) (1 - a(1 + b)a)^* = 1 - 2 a^2 - a (b+b^*) a + a(1 + b)a a(1 + b)^*a $$
$$\leq 1 - 2 a^2 - a (b+b^*) a + a(1 + b) (1 + b)^*a = 1 - a^2 + a b b^* a 
\leq 1- a^2 + a^2 =1.$$ It follows from the Gelfand--Naimark axiom that $\| 1 - a(1 + b)a\|^2=\|z\|^2 =\|z z^*\|\leq 1$.
\end{proof}

The case of JB$^*$-algebras is treated next. We first recall some notation. Given an element $a$ in a JB$^*$-algebra $B$, we shall write $U_a$ for the linear mapping on $B$ defined by $U_a(x) = 2 (a\circ x) \circ a - a^2 \circ x$ ($x\in B$). 
It is clear that if $B$ is regarded as a JB$^*$-triple with the product given in \eqref{eq triple product JBstaralg} then $U_a (x) = \{a,x^*,a\}$ for every $a,x\in B$.

\begin{lemma}\label{l norm inequality JBstr} Let $B$ be a unital JB$^*$-algebra. Suppose $a$ and $b$ are two elements in the closed unit ball of $B$ with $a$ positive. Then $$\| 1 - U_a (1 + b)\|= \| 1- \{a,1+b^*,a\}\| \leq  1.$$
\end{lemma}

\begin{proof} There is no loss of generality in assuming that $B$ is a JBW$^*$-algebra. Let us fix a unitary element $u$ in $B$. By \cite[page 294]{Wri77}, the JBW$^*$-subalgebra $\mathcal{C}$ of $B$ generated by $1,u$ and $u^*$ can be realised as a JW$^*$-subalgebra of a von Neumann algebra $A.$ In particular $\mathcal{C}$ is a commutative von Neumann algebra. Since $u$ is a unitary element in $\mathcal{C}$, we can find a hermitian element $h\in \mathcal{C}\subset B$ such that $u = e^{ih}$ (cf. \cite[Remark 10.2.2]{KadRingII1986}). Let $\widetilde{\mathcal{C}}$ denote the JB$^*$-subalgebra of $B$ generated by $1$, $a$ and $h$. A new application of \cite{Wri77} implies that $\widetilde{\mathcal{C}}$ is isometrically JB$^*$-isomorphic to a JB$^*$-subalgebra of a unital C$^*$-algebra $\widetilde{A}$. Since $a$ and $u$ are identified with elements in the unit ball of $\widetilde{A}$ with $a$ positive, Lemma \ref{l norm inequality Cstr} implies that  $$1\geq \| 1 - a(1 + u)a\|_{\widetilde{A}}  = \| 1 - U_a (1 + u)\|_{\widetilde{\mathcal{C}}}=\| 1 - U_a (1 + u)\|_{B}= \| 1- \{a,1+u^*,a\}\|_{B} .$$ Finally, by the Russo--Dye theorem for unital JB$^*$-algebras (see \cite{Sidd2010}) gives the desired conclusion.
\end{proof}

We shall next establish a JB$^*$-triple version of the previous two lemmata.

\begin{lemma}\label{l norm inequality JBstartriple} Suppose $a$ and $b$ are two elements in the closed unit ball of a JB$^*$-triple $E$. Then $$ \left\| 2 a - a^{[2]} + P_0 (a) (b)\right\| = \left\| 2 a - a^{[2]} + b -2 L(a^{[\frac12]},a^{[\frac12]}) (b) +Q(a^{[\frac12]})^2 (b) \right\| \leq  1.$$
\end{lemma}

\begin{proof} By \cite[Corollary 1]{FriRu86} we may suppose that $E$ is a  JB$^*$-subtriple of a unital JB$^*$-algebra $B$ which is an $\ell_{\infty}$-sum of a type $I$ von Neumann factor and an $\ell_{\infty}$-sum of finite dimensional simple JB$^*$-algebras. Lemma 2.3 in \cite{BuFerMarPe} implies the existence of an isometric triple embedding, $\pi : B \to B,$ such that $\pi (a)$ is a positive element in $B$. Since the elements $\pi(a)$, $1-\pi(a),$ and $-\pi(b)$ lie in the closed unit ball of $B$, we deduce from Lemma \ref{l norm inequality JBstr} that \begin{equation}\label{eq 0907 one} \| 1 - U_{1-\pi(a)} (1 - \pi(b))\|_{B}= \| 1- \{1-\pi(a),1-\pi(b)^*,1-\pi(a)\}\|_{B} \leq  1.
\end{equation}

On the other hand, it is not hard to check that, since $\pi(a)$ is positive in $B$, we have
$$2 \pi(a) - \pi(a)^{[2]} +\pi( b) -2 L(\pi(a)^{[\frac12]},\pi(a)^{[\frac12]}) (\pi(b)) +Q(\pi(a)^{[\frac12]})^2 (\pi(b))$$
$$= 2\pi(a) - \pi(a)^{[2]} + \{1-\pi(a), \pi (b)^*, 1-\pi(a)\} = 1+\{1-\pi(a), \pi (b)^*-1, 1-\pi(a)\}.$$ Finally, since $\pi$ is an isometric triple embedding we deduce that $$\left\| 2 a - a^{[2]} + b -2 L(a^{[\frac12]},a^{[\frac12]}) (b) +Q(a^{[\frac12]})^2 (b) \right\|_{E} \ \ \ \ \ \ \  \ \ \ \ \ \ \ \ \ \ \ \ \ \ \ \ \ \ \ \ \ \ \ \
$$
\begin{align*}
 &= \left\| 2 \pi(a) - \pi(a)^{[2]} +\pi( b) -2 L(\pi(a)^{[\frac12]},\pi(a)^{[\frac12]}) (\pi(b)) +Q(\pi(a)^{[\frac12]})^2 (\pi(b)) \right\|_{B} \\
 &= \left\| 1+\{1-\pi(a), \pi (b)^*-1, 1-\pi(a)\} \right\|_{B} \\
 & = \left\| 1- \{1-\pi(a), 1-\pi (b)^*, 1-\pi(a)\} \right\|_{B} \leq \hbox{(by \eqref{eq 0907 one})}\leq 1.
\end{align*}
\end{proof}

The promised characterization of those weak$^*$ closed faces in the bidual of a JB$^*$-triple $E$ corresponding to compact tripotents in $E^{**}$ can be now stated.

\begin{theorem}\label{t characterization relatively open weak* closed faces} Let $E$ be a JB$^*$-triple. Suppose $F$ is a proper weak$^*$ closed face of the closed unit ball of $E^{**}$. Then the following statements are equivalent: \begin{enumerate}[$(a)$]\item $F$ is open relative to $E$;
\item $F$ is a weak$^*$ closed face associated with a non-zero compact tripotent in $E^{**}$, that is, there exists a unique non-zero compact tripotent $u$ in $E^{**}$ satisfying that $F= F_u^{E^{**}} =u + \mathcal{B}_{E^{**}_0(u)}$.
\end{enumerate}
\end{theorem}

\begin{proof} $(b)\Rightarrow (a)$ Let us first assume that $F= F_u^{E^{**}}= u + \mathcal{B}_{E^{**}_0(u)}$, where $u$ is a  compact-$G_\delta$ tripotent in $E^{**}$, that is, $u = u(a)$ for some $a\in S(E)$. It is known that the sequence $(a^{[2n -1]})_n$ is decreasing and converges in the weak$^*$ topology (and hence in the strong$^*$ topology) of $E^{**}$ to $u(a)$.\smallskip

Pick an arbitrary $y\in F$ (i.e. $y=u +P_0(u)(y)$). Kaplansky's density theorem assures that $\mathcal{B}_{E}$ is strong$^*$ dense in $\mathcal{B}_{E^{**}}$ (cf. \cite[Corollary 3.3]{BarFri} or just apply \eqref{eq bipolar compatible}), thus  we can find a net $(y_{\lambda})_{\lambda\in \Lambda}\subset \mathcal{B}_{E}$ converging to $y$ in the strong$^*$ topology of $E^{**}$. We set $a_{n} := a^{[2n -1]}$ ($n\in \mathbb{N}$) and $$x_{\lambda,n} := \left(2 a_n - a_n^{[2]}\right) + P_0(a_n) (y_{\lambda}), \ \ \ ((\lambda,n)\in \Lambda\times \mathbb{N}). $$

We have already commented that $(a_n)_n= (a^{[2n -1]})_n\to u$ in the strong$^*$-topology of $E^{**}$. Having in mind that the triple product of $E^{**}$ is jointly  strong$^*$ continuous, and identifying the JBW$^*$-subtriple of $E^{**}$ generated by $a$ and its range tripotent, $r(a)$, with a commutative von Neumann algebra in which $a$ is a positive generator, we can easily deduce that $(a_n^{[2]})_n = (\{a_n,r(a), a_n\})_n = (\{a_n,a_n, r(a)\})_n\to u$ in the strong$^*$ topology of $E^{**}$. Moreover, the support and the range tripotents of $a$ coincides with the support and the range tripotent of $a^{[\frac14]}$, respectively, and thus $(a_n^{[\frac12]})_n = (\{ (a^{\frac12})^{[2n -1]}, r(a), (a^{\frac12})^{[2n -1]} \})_n\to u$ in the strong$^*$ topology of $E^{**}$.\smallskip

Clearly, the double indexed net $(x_{\lambda,n})_{\lambda,n}$ is contained in $E$, and by the joint strong$^*$ continuity of the triple product of $E^{**}$ the net $(x_{\lambda,n})_{\lambda,n}$ tends to $2 u - u + P_0(u) (y ) = u +P_0(u) (y) = y$ in the strong$^*$ topology of $E^{**}$, and hence in the weak$^*$ topology of the latter space.\smallskip

On the other hand, by considering the JBW$^*$-subtriple of $E^{**}$ generated by $a,$ we can easily see that $ 2 a_n - a_n^{[2]} = u + P_0(u) (2 a_n - a_n^{[2]}) \in (u + E_0^{**} (u))\cap {E}$. Since $a = u+P_0(u) (a)$, Lemma 2.5 in \cite{FerPe06} assures that $P_0(a_n) (y_{\lambda}) \in E_0^{**}(u)$. Consequently, $$ x_{\lambda,n} = 2 a_n - a_n^{[2]} + P_0(a_n) (y_{\lambda})\in (u + E_0^{**} (u))\cap {E}.$$ Lemma \ref{l norm inequality JBstartriple} proves that $x_{\lambda,n} \in \mathcal{B}_{E}$ for every $(\lambda,n)\in \Lambda\times \mathbb{N}$, and thus $x_{\lambda,n} \in F\cap E$, for every $(\lambda,n)\in \Lambda\times \mathbb{N}$. Since $(x_{\lambda,n})_{\lambda,n}\to y$ in the weak$^*$ topology of $E^{**}$, we get $y\in \overline{F\cap E}^{w^*}$. This concludes the proof in the case that $u$ is compact-$G_{\delta}$, that is $$\overline{(u + \mathcal{B}_{E^{**}_0(u)})\cap E}^{w^*} = u + \mathcal{B}_{E^{**}_0(u)},$$ for every compact-$G_{\delta}$ tripotent $u\in E^{**}$.\smallskip

Suppose now that $u$ is a non-zero compact tripotent in $E^{**}$. Then, by definition, we can find a decreasing net $(u_{\mu})_{\mu}$ of compact-$G_{\delta}$ tripotents in $E^{**}$ converging to $u$ in the weak$^*$ topology of $E^{**}$, and hence $(u_{\mu})_{\mu}\to u$ in the strong$^*$ topology. We have proved in the previous paragraphs that each $F^{E^{**}}_{u_{\mu}}$ is open relative to $E$, that is, \begin{equation}\label{eq each Fmu is open} \overline{F^{E^{**}}_{u_{\mu}}\cap E}^{w^*} = \overline{(u_{\mu} + \mathcal{B}_{E^{**}_0(u_{\mu})})\cap E}^{w^*} = F_{u_{\mu}}^{E^{**}}= u_{\mu} + \mathcal{B}_{E^{**}_0(u_{\mu})},
 \end{equation} for every $\mu$. Given an arbitrary $y\in F = F_u^{E^{**}}$, the net $( u_{\mu} + P_0(u_{\mu}) (y)) \to u+ P_0(u) (y) = y$ in the weak$^*$ topology. Since $F^{E^{**}}_{u_{\mu}}\subseteq F = F_u^{E^{**}}$ for every $\mu$, the arbitrariness of $y$ shows that \begin{equation}\label{eq weakstar closure union} F = F^{E^{**}}_{u} = \overline{\bigcup_{\mu} F^{E^{**}}_{u_{\mu}} }^{w^*}
\end{equation} Now, the relation $$\overline{F\cap E}^{w^*}\supseteq \overline{\bigcup_{\mu} \overline{F^{E^{**}}_{u_{\mu}} \cap E}^{w^*} }^{w^*} = \overline{\left(\bigcup_{\mu} F^{E^{**}}_{u_{\mu}}\right) \cap E }^{w^*}= \overline{\bigcup_{\mu} F^{E^{**}}_{u_{\mu}} }^{w^*}= \hbox{(by \eqref{eq weakstar closure union})} = F,$$ assures that $F$ is open relative to $E$.\smallskip

$(a)\Rightarrow (b)$ Since $F$ is a weak$^*$ closed face of $E^{**}$ we can find a tripotent $e\in E^{**}$ satisfying $F = F^{E^{**}}_e = e + \mathcal{B}_{E_0^{**} (e)}$ (cf. Theorem \ref{t ER weakstar closed faces}). Now, by applying that $F$ is open relative to $E$, we deduce that $G = E\cap F=(e + \mathcal{B}_{E_0^{**} (e)})\cap {E}$  is a non-empty norm closed face of $\mathcal{B}_{E}$ whose weak$^*$-closure in $E^{**}$ is $F$. Theorem \ref{thm norm closed faces} implies the existence of a non-zero compact tripotent $u\in E^{**}$ such that $G =  (u + \mathcal{B}_{E_0^{**} (u)})\cap {E}$. Finally, the implication $(b)\Rightarrow (a)$ gives $ u + \mathcal{B}_{E_0^{**} (u)} = \overline{G}^{w^*} = F = F^{E^{**}}_e = e + \mathcal{B}_{E_0^{**} (e)}$, and hence, by Theorem \ref{t ER weakstar closed faces}, $e=u$ is a non-zero compact tripotent.
\end{proof}

A particular case of the implication $(b)\Rightarrow (a)$ in Theorem \ref{t characterization relatively open weak* closed faces}, in the case in which $E=A$ is a C$^*$-algebra and $F$ is a proper weak$^*$ closed face of the closed unit ball of $A^{**}$ associated with a compact projection in $A^{**}$, is established by M. Mori and N. Ozawa in \cite[Lemma 16]{MoriOza2018}.\smallskip

We shall also need the next consequence of the above Theorem \ref{t characterization relatively open weak* closed faces}.

\begin{proposition}\label{p norm closure of the union} Let $(u_{\lambda})_{\lambda\in \Lambda}$ be a decreasing net of compact tripotents in the second dual of a JB$^*$-triple $E$. Suppose $u\neq 0$ is the infimum of the net  $(u_{\lambda})_{\lambda\in \Lambda}$ in $E^{**}$.
For each $\lambda$ in the index set, let $F^{E}_{u_{\lambda}} = (u_{\lambda}+ \mathcal{B}_{E_0^{**} (u_{\lambda})})\cap E$ and $F^{E}_{u} = (u+ \mathcal{B}_{E_0^{**} (u)})\cap E$ denote the corresponding norm closed faces of $\mathcal{B}_{E}$ associated with $u_{\lambda}$ and $u,$ respectively.  Then the identity $$F_u^E =\overline{\bigcup_{\lambda\in\Lambda} F^E_{u_\lambda} }^{\|.\|}$$ holds.
\end{proposition}

\begin{proof} Let us observe that $u$ is compact by \cite[Theorem 4.5]{EdRu96}, and $(u_{\lambda})_{\lambda}$ converges in the weak$^*$ topology of $E^{**}$ to $u$ with $u\leq u_{\lambda}$ for every $\lambda$.\smallskip

Since $u\leq u_{\lambda}$ for every $\lambda$, the containing $\displaystyle F_u^E \supset \overline{\bigcup_{\lambda\in\Lambda} F^E_{u_\lambda} }^{\|.\|}$ always holds. Arguing by contradiction, we assume the existence of $\displaystyle z_0\in F_u^E \backslash \overline{\bigcup_{\lambda\in\Lambda} F^E_{u_\lambda} }^{\|.\|}$. Since $\Lambda$ is a directed set and $(u_{\lambda})_{\lambda}$ is a decreasing net, and hence $\displaystyle F^E_{u_{\lambda_1}} \subseteq F^E_{u_{\lambda_2}}$ for every $\lambda_1\leq \lambda_2$, it is not hard to check that $\bigcup_{\lambda\in\Lambda} F^E_{u_\lambda}$ is a convex subset of $S(E)$. It follows that $\displaystyle \overline{\bigcup_{\lambda\in\Lambda} F^E_{u_\lambda} }^{\|.\|}$ is a norm closed convex subset of $S(E)$. By applying the Hahn-Banach theorem we can find a functional $\phi\in E^*$ and a positive $\delta$ satisfying \begin{equation}\label{eq conclusion HB}
\Re\hbox{e}\phi (z_0) +\delta \leq \Re\hbox{e}\phi (x), \hbox{ for all } x\in \overline{\bigcup_{\lambda\in\Lambda} F^E_{u_\lambda} }^{\|.\|}.
\end{equation}

Let $F^{E^{**}}_{u_{\lambda}}$ and $F^{E^{**}}_{u}$ be the corresponding weak$^*$ closed faces of $\mathcal{B}_{E^{**}}$ associated with ${u_{\lambda}}$ and $u$, respectively. By repeating the same arguments we gave in the second part of the proof of  $(b)\Rightarrow (a)$ in Theorem \ref{t characterization relatively open weak* closed faces} it can be established that \begin{equation}\label{eq weak* densities} F_u^{E^{**}}= \overline{F_u^{E^{**}}\cap E}^{w^*}= \overline{\left(\bigcup_{\lambda} F^{E^{**}}_{u_{\lambda}}\right) \cap E }^{w^*} = \overline{\bigcup_{\lambda} \left( F^{E^{**}}_{u_{\lambda}} \cap E\right) }^{w^*}= \overline{\bigcup_{\lambda} F^{E^{**}}_{u_{\lambda}} }^{w^*}.
 \end{equation} Having in mind that $\phi\in E^*$, we deduce from \eqref{eq conclusion HB} and from \eqref{eq weak* densities}  that $\Re\hbox{e}\phi (z_0) +\delta \leq \Re\hbox{e}\phi (z)$ for all $z\in F_u^{E^{**}},$ which is impossible because $z_0\in F_u^{E}\subseteq F_u^{E^{**}}$.
\end{proof}

\section{JBW$^*$-triples satisfying the Mazur--Ulam property}

We begin this section with a straight consequence of Corollary \ref{p complex JBW*-triples have the strong Mankiewicz property} and the facial theory of JB$^*$-triples. 
Given an element $x_0$ in an Banach space $X,$ let $\mathcal{T}_{x_0}: X\to X$ denote the translation mapping with respect to the vector $x_0$ (i.e. $\mathcal{T}_{x_0} (x) = x+x_0$, for all $x\in X$).

\begin{corollary}\label{c closed faces associated with tripotents down} Let $M$ be a JBW$^*$-triple, let $Y$ be a Banach space, and let $\Delta : S(M)\to S(Y)$ be a surjective isometry. Suppose $e$ is a non-zero tripotent in $M$, and let $F^M_e = e+ \mathcal{B}_{M_0(e)}=\left(e+ \mathcal{B}_{M^{**}_0(e)}\right)\cap M$ denote the proper norm closed face of $\mathcal{B}_{M}$ associated with $e$. Then the restriction of $\Delta$ to $F^{M}_e$ is an affine function. Furthermore, there exists an affine isometry $T_e$ from $M_0(e)$ onto a norm closed subspace of $Y$ satisfying $\Delta( \mathcal{T}_{e} (x) ) = T_e (x)$ for all $x\in \mathcal{B}_{M_0(e)}.$
\end{corollary}

\begin{proof} The arguments in \cite[Proof of Proposition 2.4 and comments after and before Corollary 2.5]{FerGarPeVill17} show that $F^M_e$ coincide with the intersection of all maximal proper norm closed faces containing it, that is, $F^{M}_e$ is an intersection face in the sense of \cite{MoriOza2018}. Therefore, by Lemma 8 in \cite{MoriOza2018}, $\Delta(F^M_e)$ also is an intersection face, and in particular a convex set.\smallskip

Let us observe that ${M_0(e)}$ is a JBW$^*$-triple and thus, by Corollary \ref{p complex JBW*-triples have the strong Mankiewicz property}, $\mathcal{B}_{M_0(e)}$ satisfies the strong Mankiewicz property. The dashed arrow in the diagram
$$\begin{tikzcd}
F^M_e \arrow{r}{\Delta|_{F^M_e}} \arrow[swap]{d}{\mathcal{T}_{-e}}
& \Delta({F}^M_e) \\
\mathcal{B}_{M_0(e)} \arrow[dashrightarrow, "\Delta_{e}"]{ru} &
\end{tikzcd}$$
defines a surjective isometry $\Delta_{e}$, which must be affine by Corollary \ref{p complex JBW*-triples have the strong Mankiewicz property}. Actually, the just quoted corollary proves the existence of a (unique) extension of $\Delta_e$ to an affine isometry $T_e$ from $M_0(e)$ onto a norm closed subspace of $Y$. The desired conclusion follows from the commutativity of the above diagram and the fact that $\mathcal{T}_{-e}$ is an affine mapping.
\end{proof}

Let us refresh our knowledge on the predual of a JBW$^*$-triple with a couple of results due to Y. Friedman and B. Russo. The first one is a consequence of \cite[Proposition 1$(a)$]{FriRu85} and reads as follows: \begin{equation}\label{eq restriction functionals to Peirce 2}\hbox{Let $e$ be a tripotent in a JB$^*$-triple $E$ and let $\varphi$ be a functional in $E^*$}
\end{equation} $$\hbox{ satisfying $\varphi (e) = \|\varphi\|$, then $\varphi = \varphi P_2(e)$.}$$
The second result tells that the extreme points in the closed unit ball of the predual, $M_*$, of a JBW$^*$-triple $M$ are in one-to-one correspondence with the minimal tripotents in $M$ via the following correspondence:
\begin{equation}\label{eq pure atoms and minimal partial isometries} \hbox{For each $\varphi\in\partial_{e}(\mathcal{B}_{M_*})$ there exists a unique minimal tripotent $e\in M$}
\end{equation} $$\hbox{satisfying $\varphi(x) e = P_2(e) (x)$ for all $x\in M$},$$ (see \cite[Proposition 4]{FriRu85}). By analogy with notation in the setting of C$^*$-algebras, the elements in $\partial_{e}(\mathcal{B}_{M_*})$ are usually called \emph{pure atoms}. For each minimal tripotent in $M$, we shall write $\varphi_e$ for the unique pure atom associated with $e$.\smallskip

The next lemma is a straight consequence of \eqref{eq restriction functionals to Peirce 2}.

\begin{lemma}\label{l the role of an abstract approximate unit} Let $\varphi$ be a normal functional in the predual of a JBW$^*$-triple $M$. Suppose $(x_{\lambda})_{\lambda}$ is a net in $M$ converging to a tripotent $e$ in the weak$^*$ topology of $M$. If $(\varphi (x_{\lambda}))_{\lambda}\to \|\varphi\|$, then $\varphi = \varphi P_2 (e)$. Consequently, if $e$ is a minimal tripotent and $\|\varphi\|=1$, then we have $\varphi = \varphi_e$.$\hfill\Box$
\end{lemma}

The following result is a quantitative version of a useful tool developed by Y. Friedman and B. Russo in \cite[Lemma 1.6]{FriRu85}. The original argument in the just quoted is combined here with \cite[Proposition 2.4]{BuFerMarPe}.

\begin{lemma}\label{l quatitative 1.6} Let $e$ be a tripotent in a JB$^*$-triple $E$, and let $x$ be an element in the closed unit ball of $E$. Then $\|P_1(e)(x)\|\leq 4\sqrt{\|e-P_2(e)(x)\|}.$
\end{lemma}

\begin{proof} By \cite[Lemma 1.1]{FriRu85} the mapping $-S_{i}(e) (\cdot) = P_2(e) - i P_1 (e) - P_0(e): E\to E$ is an isometric triple isomorphism. Set $x_j= P_j(e) (x)$ for all $j\in\{0,1,2\}$, $y= -S_{i}(e) (x)$ and $z=\frac12 (x+y)$. Clearly, $\|y\| = \|x\|\leq 1$ and $\|z\|\leq 1$ as well. We also know that $z = x_2 + \lambda x_1,$ with $\lambda = \frac{1-\iota}{2}$. By the axioms of JB$^*$-triples, $\|\{z,z,z\}\| = \|z\|^3\leq 1$, and by the contractiveness of $P_2(e)$ and Peirce arithmetic we deduce that $$\|\{x_2,x_2,x_2\} + \{x_1,x_1,x_2\} \| = \|P_2(e)\{z,z,z\}\|\leq \|\{z,z,z\}\|\leq 1.$$ Therefore $$\|e + \{x_1,x_1,e\} \|\leq \|\{e,e,e\}- \{x_2,x_2,x_2\} \| + \|\{x_1,x_1,e\} - \{x_1,x_1,x_2\} \|$$ $$ + \| \{x_2,x_2,x_2\} + \{x_1,x_1,x_2\} \|\leq 4\|e-x_2\|+1. $$ Having in mind that $\{x_1,x_1,e\}$ is a positive element in the JB$^*$-algebra $E_2(e)$ and $e$ is its unit (cf. \cite[Lemma 1.5]{FriRu85}), we get $$1 +\|\{x_1,x_1,e\} \| =\|e + \{x_1,x_1,e\} \|\leq 4\|e-x_2\|+1.$$ Finally \cite[Proposition 2.4]{BuFerMarPe} gives $\|x_1\| \leq 2 \sqrt{\|\{x_1,x_1,e\} \|} \leq 4 \sqrt{\|e-x_2\|}.$

\end{proof}


We next prove the existence of approximate units for elements in the face associated with compact-G$_\delta$ tripotents in the bidual of a JBW$^*$-triple.

\begin{lemma}\label{l norm convergence like approximae units} Let $M$ be a JBW$^*$-triple, and let $u$ be a compact-G$_\delta$ tripotent in $M^{**}$ associated with a norm-one element $a\in M$. Then there exists a decreasing sequence of non-zero tripotents  $(e_n)_{n}$ in $M$ {\rm(}actually in the JBW$^*$-subtriple of $M$ generated by $a${\rm)} satisfying that for each $x\in F_u^M$ the sequence $\Theta_n (x):=e_n + P_0(e_n) (x)$ converges to $x$ in the norm topology of $M$.
\end{lemma}

\begin{proof} By the assumptions $u=u(a)$ is the support tripotent of $a$ in $M^{**}$. 
It is known that the JBW$^*$-triple, $W_{a}$, of $M$ generated by the element $a$ is isometrically JBW$^*$-triple isomorphic to a commutative von Neumann algebra $W$ admitting $a$ as a positive generator (cf. \cite[Lemma 3.11]{Horn87} and \cite{Ka83}).\smallskip

By the Borel functional calculus in $W\cong W_a$, we set $e_n = \chi_{(1-\frac1n,1]}(a)\in W_a$ ($n\in\mathbb{N})$. Clearly, $(e_n)_n$ is a decreasing sequence of tripotents in $W_a\subset M$.\smallskip

We fix $x\in F_u^M$. Let us insert some notation. The symbol $\mathcal{D}$ will stand for the set of all continuous functions $f:[0,1]\to [0,1]$ satisfying $f(0)=0$ and $f(1)=1$.  Let $\mathcal{C}$ be the subset of $M$ given by $$\mathcal{C} = \left\{ L(f_t(a)^{[\frac12]},f_t(a)^{[\frac12]})(x)-f_t(a) : f\in \mathcal{D} \right\}.$$

We claim that $\mathcal{C}$ is a convex set. To prove the claim let $r=r_{_{M^{**}}}(a)$ denote the range tripotent of $a$ in $M^{**}$, and let $\pi$ be a linear isometric triple homomorphism from $M^{**}$ into a JBW$^*$-algebra
$B$ such that $\pi(r)$ is a projection in $B$ and $\pi |_{M_2^{**}(r)} :M_2^{**}(r)\to  \pi(M^{**})_2(\pi(r))$ is a unital Jordan $^*$-monomorphism (cf. \cite[Lemma 3.9]{EdFerHosPe2010} or \cite[Lemma 2.3]{BuFerMarPe}). It is not hard to see that \begin{align}\label{eq first triple embedding for convexity} \pi \left( L(f_t(a)^{[\frac12]},f_t(a)^{[\frac12]})(x)-f_t(a) \right) \\
=& L(f_t(\pi(a))^{[\frac12]},f_t(\pi(a))^{[\frac12]})(\pi(x))-f_t(\pi(a)) \nonumber\\
=& L(f(\pi(a))^{\frac12},f(\pi(a))^{\frac12})(\pi(x))-f(\pi(a)) \nonumber\\
=& f(\pi(a))\circ \pi(x) -f(\pi(a))\nonumber
\end{align}  where $f(a)$ denotes the continuous functional calculus of the JBW$^*$-algebra $B$ at the element $\pi(a)$. The above observation implies that $$\pi (\mathcal{C}) =\left\{  f(\pi(a))\circ \pi(x) -f(\pi(a)) : f\in \mathcal{D} \right\},$$ and thus $\pi (\mathcal{C})$ (and hence $\mathcal{C}$) is a convex set because $\mathcal{D}$ is.\smallskip

It is well known that $u$ lies to the strong$^*$-closure in $M^{**}$ of the set $\mathcal{D}(a):=\{f_t(a) : f\in\mathcal{D}\}$. Thus, we can find a net $(a_\lambda)_{\lambda}$ in $\mathcal{D}(a)$ such that $(a_\lambda^{[\frac12]})_{\lambda}$ converges to $u$ in the strong$^*$ topology of $M^{**}$. Having in mind that the triple product of $M^{**}$ is jointly strong$^*$ continuous on bounded sets, it follows that $$L(a_\lambda^{[\frac12]},a_\lambda^{[\frac12]})(x)-a_\lambda \to L(u,u)(x)-u=0,$$ in the strong$^*$ topology of $M^{**}$. Therefore $0\in \overline{\mathcal{C}}^{\hbox{\tiny strong}^*}$. Since the strong$^*$-topology of $M^{**}$ is compatible with the duality $(M^{**},M^*)$, and $\mathcal{C}$ is a convex subset of $M$, the closure of $\mathcal{C}$ in the strong$^*$ topology coincides with its weak$^*$ closure in $M^{**}$ (compare \eqref{eq bipolar compatible}). Furthermore, $0\in \overline{\mathcal{C}}^{\hbox{\tiny strong}^*}= \overline{\mathcal{C}}^{ w^*}$ assures that $0$ belongs to the weak closure of $\mathcal{C}$ in $M$, and the latter with the norm closure of $\mathcal{C}$ in $M$. We can therefore conclude that \begin{equation}\label{eq 0 is the norm closure}\hbox{ $0$ lies in the norm closure of the set $\mathcal{C}$ in $M$.}
\end{equation}

Given an arbitrary $1>\varepsilon>0$, by \eqref{eq 0 is the norm closure}, we can find an element $d= f_t(a)$ with $f\in \mathcal{D}$ such that $\|L(d^{[\frac12]},d^{[\frac12]})(x)-d\|<\frac{\varepsilon^2}{128}$. Clearly, $d$ and $d^{[\frac12]}$ belong to the face $F_u^M$ (cf. \cite[Lemma 3.3]{EdFerHosPe2010}).\smallskip

Let us be more precise. By considering the JB$^*$-triple, $M_a$, of $M$ generated by $a$, and its identification with a commutative C$^*$-algebra of the form $C_0(\Omega_a)$, where $\Omega_a\subseteq [0,1]$, with $\Omega_a\cup\{0\}$ compact and correspond to the function $t\mapsto t$ for every $t\in \Omega_a$  (cf. \cite[Corollary 1.15]{Ka83} and \cite{FriRu85}). Having in mind that $d=f_t(a)$ with $f\in \mathcal{D}$, we can find $t_0\in  [0,1)$ such that $1-\frac{\varepsilon^2}{256} \leq f(t) \leq 1$ for all $t\in [t_0,1]$. There exits a natural $n_0$ satisfying $1-\frac{1}{n_0} > t_0$. Let $g:[0,1]\to [0,1]$ be the continuous function defined by
$$g(t)= \left\{%
\begin{array}{ll}
        f(t), &  0 \leq t \leq t_0, \\
        \hbox{affine} & t_0 \leq t\leq 1-\frac{1}{n_0},\\
    1, &  1-\frac{1}{n_0}\leq t\leq 1, \\
     \end{array}%
\right. $$ and define $c= g_t(d)\in M_a$. Clearly, $\|c-d\|\leq \frac{\varepsilon^2}{256}$ and $e_n\leq e_{n_0} \leq c \leq c^{[\frac12]}$ in $W_a = \overline{M_a}^{\sigma(M,M_*)}$ for all $n\geq n_0$. Having in mind the isometric triple embedding $\pi$, it follows, as in \eqref{eq first triple embedding for convexity}, that
$$\|L(c^{[\frac12]},c^{[\frac12]})(x)-c\| = \|\pi L(c^{[\frac12]},c^{[\frac12]})(x)-\pi (c)\| = \|\pi(c)\circ \pi(x) -\pi(c)\|$$
$$\leq \| (\pi(c)-\pi(d))\circ \pi(x) - (\pi(c)-\pi(d))\|+  \|\pi(d)\circ \pi(x) -\pi(d)\| $$ $$ \leq 2\|\pi(c)-\pi(d)\| + \|\pi(d)\circ \pi(x) -\pi(d)\|= 2 \|c-d\| + \|L(d^{[\frac12]},d^{[\frac12]})(x)-d\|<\frac{\varepsilon^2}{64}.$$

Since $e_n\leq c\leq  c^{[\frac12]}$ in $W_a $ for all $n\geq n_0$, we deduce that $c = e_n + P_0(e_n) (c)$ and $c^{[\frac12]} = e_n + P_0(e_n) (c^{[\frac12]})$ for all $n\geq n_0$. Therefore, by applying Peirce arithmetic, we have  $P_2(e_n) ( L(c^{[\frac12]},c^{[\frac12]})(x)-c) = P_2(e_n) (x) -e_n,$ which combined with the contractiveness of $P_2(e_n)$ gives $$\| P_2(e_n) (x) -e_n\|<\frac{\varepsilon^2}{64},\ \hbox{ for all } n\geq n_0.$$ Lemma \ref{l quatitative 1.6} assures that $$ \| P_1(e_n) (x)\| \leq 4 \sqrt{\| P_2(e_n) (x) -e_n\|} < 4 \sqrt{\frac{\varepsilon^2}{64}} < \frac{\varepsilon}{2}.$$

Finally we compute the distance between $\Theta_n(x)=e_n+P_0(e_n) (x)$ and $x$. In this case, for each $n\geq n_0$ we have $$\| \Theta_n(x)-x\|=\|e_n+P_0(e_n) (x) -x\| =\|e_n-P_2(e_n) (x) -P_1(e_n) (x)\| $$
$$\leq \|e_n-P_2(e_n) (x)\| + \|P_1(e_n) (x)\|< \frac{\varepsilon^2}{64}+ \frac{\varepsilon}{2} < \varepsilon,$$ for every $n\geq n_0$.
\end{proof}

In the next results we shall establish a version of the conclusions in \cite[Lemma 15]{MoriOza2018} in the setting of JBW$^*$-triples.

\begin{proposition}\label{p affinity on closed Gdelta faces} Let $M$ be a JBW$^*$-triple, let $Y$ be a Banach space, and let $\Delta : S(M)\to S(Y)$ be a surjective isometry. Then the restriction of $\Delta$ to each norm closed {\rm(}proper{\rm)} face of $\mathcal{B}_{M}$ associated with a compact-$G_\delta$ tripotent $u$ in $M^{**}$ is an affine function. Furthermore, for each $\psi\in Y^*$, there exist $\phi\in M^*$ and $\gamma\in \mathbb{R}$ such that $\|\phi\|,|\gamma|\leq  \|\psi\|,$ and $$\psi \Delta(x) = \Re\hbox{e}\, \phi(x) + \gamma,\hbox{ for all $x\in F_u^M$.}$$
\end{proposition}

\begin{proof} Let $a$ be a norm-one element in $M$, and let $u=u(a)$ be the support tripotent of $a$ in $M^{**}$. By Lemma \ref{l norm convergence like approximae units} there exists a sequence of non-zero tripotents $(e_n)_{n}$ in $M$ satisfying that for each $x\in F_u^M$ the sequence $\Theta_n (x):=e_n + P_0(e_n) (x)$ converges to $x$ in the norm topology of $M$. Clearly, $\Theta_n(y)\in F_{e_n}^M$ for all $n\in \mathbb{N}$, $y\in M$.\smallskip

Now, take $x,y\in F^{M}_{u(a)}$ and $t\in ]0,1[$. Since each $\Theta_n$ is an affine map and $\Delta|_{F^{M}_{e_n}}$ is also affine (see Corollary \ref{c closed faces associated with tripotents down}), we deduce that $$\Delta (\Theta_n(t x + (1-t) y )) = \Delta (t \Theta_n( x)  + (1-t) \Theta_n(y)) =  t \Delta (\Theta_n ( x) ) + (1-t) \Delta (\Theta_n(y )),$$ for every $n\in \mathbb{N}$. Taking limits in $n\to \infty$, it follows from Lemma \ref{l norm convergence like approximae units} and from the norm continuity of $\Delta$ that $\Delta (t  x  + (1-t) y) = t \Delta ( x ) + (1-t) \Delta (y),$ which proves that $\Delta|_{F_u^M}$ is affine.\smallskip

For the last assertion, let us fix $\psi\in Y^*$. By Corollary \ref{c closed faces associated with tripotents down}, for each natural $n$, we can find a linear isometry $T_n:M_0(e_n)\to Y$ and a norm-one element $y_n=\Delta(e_n)\in S(Y)$ such that $\Delta (w) = T_n (w-e_n) + y_n$, for all $w\in F_{e_n}^M$. Let us define $\Re\hbox{e}\,\phi_n = \psi T_n P_0(e_n)\in (M_{\mathbb{R}})^*$, with $\phi_n \in M^*$ and $\|\phi_n\|\leq \|\psi\|$. We can therefore write \begin{equation}\label{eq 1 2501} \psi \Delta (w) =  \Re\hbox{e}\,\phi_n (w) + \gamma_n,
\end{equation} for all $w\in F_{e_n}^M$, where $\gamma_n :=\psi(y_n)$ and $|\gamma_n|\leq \|\psi\|$. Find a subsequence $(\gamma_{\sigma(n)})_n$ converging to some $\gamma\in \mathbb{R}$ with $|\gamma|\leq \|\psi\|$. The sequence $(\phi_{\sigma(n)})_n$ is bounded in $\mathcal{B}_{M^*}$. Let $\phi\in M^*$ be a $\sigma(M^*,M)$-cluster point of $(\phi_{\sigma(n)})_n$ with $\|\phi\|\leq \|\psi\|$.\smallskip

Take now an element $x\in F_{u}^{M}$. We deduce from Lemma \ref{l norm convergence like approximae units} and the continuity of $\Delta$ that $\psi\Delta(\Theta_{\sigma(n)}(x))\to \psi\Delta (x)$ in $\mathbb{R}$. It follows from \eqref{eq 1 2501} that $$\psi \Delta \Theta_{\sigma(n)}(x) =  \Re\hbox{e}\,\phi_{\sigma(n)} \Theta_{\sigma(n)}(x) + \gamma_{\sigma(n)},$$ for all natural $n$. Since $\phi\in M^*$ is a $\sigma(M^*,M)$-cluster point of $(\phi_{\sigma(n)})_n$ and $\|\Theta_{\sigma(n)} (x)-x\|\to 0$, we conclude that $( \Re\hbox{e}\,\phi_{\sigma(n)} \Theta_{\sigma(n)}(x))_n \to \Re\hbox{e}\, \phi(x)$. By combining all these assertions we get  $\psi \Delta(x) = \Re\hbox{e}\, \phi(x) + \gamma.$
\end{proof}

By applying Proposition \ref{p affinity on closed Gdelta faces}, we can now deal with general proper norm closed faces in the closed unit ball of a JBW$^*$-triple.

\begin{proposition}\label{p affinity on closed faces} Let $M$ be a JBW$^*$-triple, let $Y$ be a Banach space, and let $\Delta : S(M)\to S(Y)$ be a surjective isometry. Then the restriction of $\Delta$ to each norm closed proper face $F$ of $\mathcal{B}_{M}$ is an affine function. Furthermore, for each $\psi\in Y^*$, there exist $\phi\in M^*$ and $\gamma\in \mathbb{R}$ such that $\|\phi\|,|\gamma|\leq \|\psi\|,$ and $$\psi \Delta(x) = \Re\hbox{e}\, \phi(x) + \gamma,\hbox{ for all $x\in F$.}$$
\end{proposition}

\begin{proof} Let $F$ be a proper norm closed face of $\mathcal{B}_M$. We know from Theorem \ref{thm norm closed faces} that $F= F^{M}_{u},$ where $u$ is a compact tripotent in $M^{**}.$ Then there exists a net $(u_{\lambda})_{\lambda\in \Lambda}$ of compact-$G_\delta$ tripotents in $M^{**}$ decreasing in the weak$^*$ topology of $M^{**}$ to $u$ (cf. \cite{EdRu96}). 
For each $\lambda\in \Lambda$ we write $F^{M}_{u_\lambda} = \left(u_\lambda+ \mathcal{B}_{M^{**}_0(u_\lambda)}\right)\cap M$ for the proper norm closed face associated with $u_\lambda$.\smallskip

Proposition \ref{p norm closure of the union} assures that $\displaystyle F= F^M_u =\overline{\bigcup_{\lambda\in\Lambda} F^M_{u_\lambda} }^{\|.\|}.$ For each $\lambda\in \Lambda$, $u_\lambda$ is a compact-$G_\delta$ tripotent in $M$, and thus Proposition \ref{p affinity on closed Gdelta faces} implies that the restriction of $\Delta$ to the face $F^M_{u_\lambda} = \left(u_\lambda+ \mathcal{B}_{M^{**}_0(u_\lambda)}\right)\cap M$ is an affine function. \smallskip

Now, taking $\displaystyle x,y\in \bigcup_{\lambda\in\Lambda} F^M_{u_\lambda}$ and $t\in ]0,1[$, we can find $\lambda_0\in \Lambda$ such that $x,y\in F^M_{u_{\lambda_0}}$. By applying that $\Delta|_{F^M_{u_{\lambda_0}}}$ is an affine mapping, we deduce that $$ \Delta (t x + (1-t) y ) =   t \Delta (x ) + (1-t) \Delta (y).$$ This proves that $\displaystyle \Delta|_{\bigcup_{\lambda\in\Lambda} F^M_{u_\lambda}}$ is affine. The norm-density of $\bigcup_{\lambda\in\Lambda} F^M_{u_\lambda}$ in $F$ and the continuity of $\Delta$ can be now applied to deduce that $\Delta|_{F}$ is affine.\smallskip

Let us prove the final statement. For this purpose we fix $\psi\in Y^*\backslash\{0\}$. For each $\lambda \in \Lambda$, Proposition \ref{p affinity on closed Gdelta faces} implies the existence of a functional $\phi_{\lambda}\in M^*$ and $\gamma_{\lambda}\in \mathbb{R}$ such that  $\|\phi_{\lambda}\|,|\gamma_{\lambda}|\leq \|\psi\|,$ and \begin{equation}\label{eq 2701 lambda} \psi \Delta(x) = \Re\hbox{e}\, \phi_{\lambda}(x) + \gamma_{\lambda},\hbox{ for all $x\in F_{u_{\lambda}}^M$, $\lambda\in \Lambda$.}
 \end{equation}By the weak$^*$ compactness of $\mathcal{B}_{M^*}$ we can find common subnets $(\phi_{\mu})$ and $(\gamma_{\mu})$ converging to $\phi\in M^*$ and $\gamma\in \mathbb{R}$, respectively. Clearly $\|\phi\|\leq \|\psi\|$ and $|\gamma|\leq \|\psi\|$. We claim that $$\psi \Delta(x) = \Re\hbox{e}\, \phi(x) + \gamma,\hbox{ for all $x\in F$.}$$ Namely, for each $\varepsilon>0$, we can find $\mu_0$ and $x_{\mu_0}\in F_{u_{\mu_0}}^M$ such that $\|x-x_{\mu_0}\|<\frac{\varepsilon}{6 \|\psi\|}$, $|\gamma_{\mu_0}-\gamma|<\frac{\varepsilon}{3}$ and $|\phi(x)-\phi_{\mu_0} (x)|<\frac{\varepsilon}{3}$. We therefore conclude from \eqref{eq 2701 lambda} that $$ \left| \psi \Delta(x) - \Re\hbox{e}\, \phi(x) - \gamma \right|\leq \left| \psi \Delta(x) -\psi \Delta(x_{\mu_0}) \right| + \left| \psi \Delta(x_{\mu_0}) - \Re\hbox{e}\, \phi_{\mu_0}(x_{\mu_0}) - \gamma_{\mu_0} \right|$$
$$+ \left| \Re\hbox{e}\, \phi_{\mu_0}(x_{\mu_0}) - \Re\hbox{e}\, \phi_{\mu_0}(x) \right| + \left| \Re\hbox{e}\, \phi_{\mu_0}(x) - \Re\hbox{e}\, \phi(x) \right| +|\gamma_{\mu_0}-\gamma|$$
$$\leq \left(\|\psi\|+\|\phi_{\mu_0}\|\right) \, \|x_{\mu_0}-x\| + 2\frac{\varepsilon}{3} < \varepsilon.$$ The desired statement follows from the arbitrariness of $\varepsilon$.
\end{proof}

We can now mimic the ideas in \cite{Di:C}, \cite{Liu2007}, \cite[Lemma 2.1]{JVMorPeRa2017}, and \cite[Lemma 2.1]{CuePer} to prove the existence of support functionals for faces.

\begin{lemma}\label{l existence of support functionals for the image of a face} Let $E$ be a JB$^*$-triple and let $Y$ be a real Banach space. Suppose $\Delta : S(E)\to S(Y)$ is a surjective isometry. Then for each maximal proper norm closed face $F$ of the closed unit ball of $E$ the set $$\hbox{supp}_{\Delta}(F) := \{\psi\in Y^* : \|\psi\|=1,\hbox{ and } \psi^{-1} (\{1\})\cap \mathcal{B}_{Y} = \Delta(F) \}$$ is a non-empty weak$^*$ closed face of $\mathcal{B}_{Y^*}$; in other words, for each minimal tripotent $e$ in $E^{**}$ the set $$ \hbox{supp}_{\Delta}(F_e^{E}) := \{\psi\in Y^* : \|\psi\|=1,\hbox{ and } \psi^{-1} (\{1\})\cap \mathcal{B}_{Y} = \Delta(F_e^E) \}$$ is a non-empty weak$^*$ closed face of $\mathcal{B}_{Y^*}.$
\end{lemma}

\begin{proof} By applying \cite[Lemma 5.1$(ii)$]{ChenDong2011} (see also \cite[Lemma 3.5]{Tan2014}) we deduce that the set $\Delta(F)$ is a maximal convex subset of $Y$. It follows from  Eidelheit's separation theorem \cite[Theorem 2.2.26]{Megg98} that there exists a norm-one functional $\varphi\in Y^*$ such that $\varphi^{-1} (\{1\})\cap \mathcal{B}_{Y} = \Delta(F)$ (compare the proof of \cite[Lemma 3.3]{Tan2016} or \cite[Lemma 2.1]{CuePer}).
\end{proof}

For the sake of brevity and conciseness, we introduce the following notation.

\begin{definition}\label{def property P} Let $E$ be a JB$^*$-triple. We shall say that $E$ satisfies property {\rm(}$\mathcal{P}${\rm)} if for each minimal tripotent $e$ in $E^{**}$ and each complete tripotent $u$ in $E$ {\rm(}that is $u\in \partial_e(\mathcal{B}_{E})${\rm)}, there exists another minimal tripotent $w$ in $E^{**}$ satisfying $w\perp e$ and $u = w + P_0(w) (u)$.
\end{definition}

Another tool needed in the proof of our main result is established in the next result.

\begin{proposition}\label{p technical before theorem} Let $M$ be a JBW$^*$-triple satisfying property {\rm(}$\mathcal{P}${\rm)}. Let $\varphi_e\in\partial_{e}(\mathcal{B}_{M^*})$ denote the unique pure atom associated with a minimal tripotent $e$ in $M^{**}$. Suppose $\Delta : S(M) \to S(Y)$ is a surjective isometry from the unit sphere of $M$ onto the unit sphere of a real Banach space $Y$. Then for each $\psi$ in $\hbox{supp}_{\Delta}(F_e^{M})$ we have $\psi \Delta (u) = \Re\hbox{e}\varphi_e(u)$ for every non-zero tripotent $u$ in $\partial_e(\mathcal{B}_{M})$.
\end{proposition}

\begin{proof}  Let us fix a minimal tripotent $e$ in $M^{**}$, $u\in \partial_e(\mathcal{B}_{M}),$ and $\psi$ in $\hbox{supp}_{\Delta}(F_e^{M})$. By the hypotheses on $M$ we can find another minimal tripotent $w$ in $M^{**}$ satisfying $w\perp e$ and $u = w + P_0(w) (u)$. Proposition \ref{p affinity on closed faces} implies the existence of $\lambda_w\in \mathbb{R}$ and $\varphi\in M^*$ such that $\|\varphi\|\leq 1$ and $\psi \Delta (x) = \lambda_w+ \Re\hbox{e} \varphi (x)$ for every $x\in F_w^M$.\smallskip

Since minimal tripotents in $M^{**}$ are compact, we are in a position to apply the non-commutative generalisation of Urysohn's lemma established in \cite[Proposition 3.7]{FerPe10b}. By this result, we can find orthogonal norm-one elements $a_0,b_0\in M$ such that $a_0 = e+P_0(e) (a_0)$ and $b_0 = w+P_0(w) (b_0),$ that is, $a_0\in F_e^{M}$ and $b_0\in F_w^{M}$. Since, by orthogonality, $\pm a_0 + b_0\in \left(\pm F_e^{M}\right)\cap F_w^{M},$ we deduce from Lemma \ref{l existence of support functionals for the image of a face} and \cite[Lemma 8]{MoriOza2018} that $$ \pm 1 = \psi \Delta (\pm a_0 + b_0) = \lambda_w\pm \Re\hbox{e} \varphi (a_0) + \Re\hbox{e} \varphi (b_0) = \pm \Re\hbox{e} \varphi (a_0) + \psi \Delta ( b_0),$$ which implies that $\psi \Delta ( b_0)= 0$ and $\Re\hbox{e} \varphi (a_0)=1$. In the above argument, $a_0$ can be arbitrarily replaced with any element $c$ in the face $F_e^{M}$ for which there exists $b_0\in F_w^{M}$ with $c\perp b_0$. Arguing as in the proof of \cite[Lemma 2.7]{AkPed92} we can find a net $(a_{\lambda})$ in $M_2(r_{M}(a_0))\subseteq M$ such that $a_\lambda = e+P_0(e) (a_\lambda)$ (equivalently $a_\lambda\in F_e^{M}$), and $(a_{\lambda})\to e$ in the weak$^*$ topology of $M^{**}$. Since $a_{\lambda}\perp b_0$ for every $\lambda$, it follows from the above arguments that $\Re\hbox{e} \varphi (a_\lambda)=1=\|\varphi\|$ for all $\lambda$. Lemma \ref{l the role of an abstract approximate unit} assures that $\varphi=\varphi_e$.\smallskip

We have therefore shown that $\psi \Delta (x) = \lambda_w+ \Re\hbox{e} \varphi_e (x)$ for every $x\in F_w^M$. Since $b_0\in F_w^M$ and $b_0\perp e$, we get $0= \psi \Delta (b_0) = \lambda_w+ \Re\hbox{e} \varphi_e (b_0),$ which implies that $\lambda_w=0,$ and $\psi \Delta (u) = \Re\hbox{e} \varphi_e (u)$ as desired.
\end{proof}

Let us recall another result proved by M. Mori and N. Ozawa in \cite[Lemma 18]{MoriOza2018}. Let $A$ be a unital C$^*$-algebra, and let $p$ be a minimal projection in $A^{**}$. Then for each $a\in F_p^A = \left(p + A_0^{**} (p) \right)\cap \mathcal{B}_{A}$ and each $\varepsilon >0$ there exist unitary elements $u_1,\ldots, u_m$ in $F_p^A$ and $t_1,\dots, t_m\in [0,1]$ satisfying  $\displaystyle \sum_{j=1}^m t_j =1$ and $\displaystyle \left\|a - \sum_{j=1}^m t_j u_j\right\|< \varepsilon.$ We shall establish a version of these results in the setting of JBW$^*$-triples.

\begin{lemma}\label{l MO L18 for JBWtriple} Let $e$ be a non-zero compact tripotent in the second dual of a JBW$^*$-triple $M$. Let $a$ be an element in the norm closed face $F_e^M$. Then $a$ can be written as the average of two extreme points of $\mathcal{B}_M$ belonging to the face $F_e^M$.
\end{lemma}

\begin{proof} Let $a\in F_e^M =\left({\{e\}}_{\prime}\right)_{\prime}$. By \cite[Theorem 5]{Sidd2007} $a$ can be written in the form $a = \frac{u_1+u_2}{2},$ where $u_1,u_2\in \partial_e(\mathcal{B}_{M})$. Let us pick an arbitrary $\varphi \in \{e\}_{\prime}\subset S(M^*)$. Since $1= \varphi (a) = \frac{\varphi(u_1)+\varphi(u_2)}{2},$ $\varphi(u_1)=\varphi(u_2)=1,$ it follows from the arbitrariness of $\varphi \in \{e\}_{\prime}\subset S(M^*)$ that  $u_1,u_2\in  \left({\{e\}}_{\prime}\right)_{\prime} = F_e^M$ as desired.
\end{proof}

Now, by combining Proposition \ref{p technical before theorem} and Lemma \ref{l MO L18 for JBWtriple} we get the following result.

\begin{corollary}\label{c technical before theorem 2} Let $M$ be a JBW$^*$-triple satisfying property {\rm(}$\mathcal{P}${\rm)}. Let $\varphi_e\in\partial_{e}(\mathcal{B}_{M^*})$ denote the unique pure atom associated with a minimal tripotent $e$ in $M^{**}$. Suppose $\Delta : S(M) \to S(Y)$ is a surjective isometry from the unit sphere of $M$ onto the unit sphere of a real Banach space $Y$. Then for each $\psi$ in $\hbox{supp}_{\Delta}(F_e^{M})$ we have $\psi \Delta (x) = \Re\hbox{e}\varphi_e(x)$ for every $x\in S(M)$.
\end{corollary}

\begin{proof} Let $e$ be a minimal tripotent in $M^{**}$, $\varphi_e\in\partial_{e}(\mathcal{B}_{M^*})$ the unique pure atom associated with $e$, and let $\psi$ be an element in $\hbox{supp}_{\Delta}(F_e^{E})$. Proposition \ref{p technical before theorem} implies that $\psi \Delta (u) =\Re\hbox{e} \varphi_e (u)$ for every $u\in \partial_{e}(\mathcal{B}_{M})$.\smallskip

Let us fix $x\in S(M)$. By applying Zorn's lemma 
there exists a minimal tripotent $v\in M^{**}$ such that $x\in F_v^M = \left(v+M_0^{**}(v)\right)\cap \mathcal{B}_{M}$. Lemma \ref{l MO L18 for JBWtriple} and Proposition \ref{p affinity on closed faces} give the desired statement.
\end{proof}

Before approaching our main goal we shall establish a technical result. Let $e$ and $v$ be two tripotents in a JB$^*$-triple $E$. Accordingly to the standard notation, we shall say that $e$ and $v$ are \emph{collinear} if $e\in E_1(v)$ and $v\in E_1(e)$.\smallskip

Y. Friedman and B. Russo proved in \cite[Proposition 6]{FriRu85} that every JBW$^*$-triple $M$ satisfies a pre-variant of the so-called \emph{``extreme ray property''}, that is,  if $u$ and $e$ are tripotents in $M$ and $e$ is minimal, then $P_2(u) (e)$ is a scalar multiple of another minimal tripotent in $M$. Actually, the same conclusion holds when $M$ is a JB$^*$-triple $E$ because minimal (respectively, complete) tripotents in $E$ are minimal (respectively, complete) in $E^{**}$.

\begin{lemma}\label{l minimal respect complete}
Let $u$ be a complete tripotent in a JB$^*$-triple $E$, that is, $u\in \partial_e(\mathcal{B}_E)$. Then every minimal tripotent $e$ in $E$ decomposes as a linear combination of the form $e=\lambda v+\mu w$, where $v$ and $w$ are two collinear tripotents in $E$ which are  minimal or zero, $v\in E_2(u)$, $w\in E_1(u)$, and $\lambda,\mu\in \mathbb{R}^+_0$ satisfy $\lambda^2+\mu^2=1$.
\end{lemma}

\begin{proof} By applying that $u$ is a complete tripotent (i.e. $E_0(u)=\{0\}$), we deduce that $e =e_2+e_1$ where $e_k = P_k(u) (e)\in E_k(u)$ for $k=1,2$.\smallskip

Since $e$ is minimal, the pre-variant version of the extreme ray property (see \cite[Proposition 6]{FriRu85}) implies that $P_2 (u)(e) = e_2 = \lambda v$, where $v$ is a minimal tripotent in $E$ and $\lambda\geq 0$. We observe that $e=e_1$ is a minimal tripotent whenever $\lambda$ vanishes. We can thus assume that $\lambda> 0$. It is clear that $e_1$ belongs to $E_1(v).$ Since by the identity
$$\alpha e= \{e,e_2,e\}=\{e_2,e_2,e_2\}+2\{e_2,e_2,e_1\}+\{e_1,e_2,e_1\} \ \hbox{(with $\alpha\in \mathbb{C})$},$$ combined with Peirce rules and the completeness of $u$, we obtain that $\{e_1,e_2,e_1\}=0$, $\alpha \lambda v=\{e_2,e_2,e_2\} = |\lambda|^2 \lambda v$ (and thus $\alpha = |\lambda|^2$), and $|\lambda|^2 e_1 = 2\{e_2,e_2,e_1\} = 2 |\lambda|^2 \{v,v,e_1\},$ which proves that $e_1\in E_1(v).$\smallskip

Now, having in mind the identity
$$\gamma e= \{e,e_1,e\}=\{e_1,e_1,e_1\}+2\{e_2,e_1,e_1\}+\{e_2,e_1,e_2\} \ \hbox{(with $\gamma\in \mathbb{C})$},$$ we get $ \{e_1,e_1,e_1\} = \gamma e_1$, and hence, by the triple functional calculus, $e_1$ is a multiple of a tripotent in $E$. That is, $e_1 = \mu w$, where $w$ is a tripotent in $E$, $\mu\geq 0$, and $|\mu|^2 = \gamma$. We may reduce to the case in which $\mu \neq 0$.  We further know that $2 \lambda \mu^2 \{v,w,w\}=2\{e_2,e_1,e_1\} = \gamma e_2 = \gamma \lambda v,$ witnessing that $2 \{v,w,w\} = v$. Therefore $v$ and $w$ are collinear tripotents in $E$. The \emph{triple system analyzer} (see \cite[Proposition 2.1 and Lemma in page 306]{DanFri87}) gives the desired statement.
\end{proof}

The best known examples of JBW$^*$-triples are given by the so-called Cartan factors. There are six types of Cartan factors defined as follows:\smallskip

\emph{Cartan factor of type 1}: the complex Banach space $B(H, K)$, of all bounded linear operators between two complex Hilbert spaces, $H$ and $K$, whose triple product is given by \eqref{eq product operators}.\smallskip

Given a conjugation, $j$, on a complex Hilbert space, $H$, we can define a linear involution on $B(H)$ defined by $x \mapsto x^{t}:=j x^* j$. \smallskip

\emph{Cartan factor of type 2:} the subtriple of $B(H)$ formed by the skew-symmetric operators for the involution $t$.\smallskip

\emph{Cartan factor of type 3:} the subtriple of $B(H)$ formed by the $t$-symmetric operators.\smallskip

\emph{Cartan factor of type 4} or \emph{spin:} a complex Banach space $X$ admitting a complete inner product $(.|.)$ and a conjugation $x\mapsto \overline{x},$ for which the norm of $X$ is given by $$ \|x\|^2 = (x|x) + \sqrt{(x|x)^2 -|
(x|\overline{x}) |^2}.$$

\emph{Cartan factors of types 5 and 6} (also called \emph{exceptional} Cartan factors) consist of matrices over the eight
dimensional complex algebra of Cayley numbers; the type 6 consists of all 3 by 3 self-adjoint matrices and has a natural Jordan algebra structure, and the type 5 is the subtriple consisting of all 1 by 2 matrices.\smallskip

Our next goal is to show the connection between rank and property {\rm(}$\mathcal{P}${\rm)} in the case of Cartan factors.

\begin{proposition}\label{p Cartan factors fo rank bigger than or equal to three satisfy P } Every Cartan factor of rank bigger than or equal to three satisfies property {\rm(}$\mathcal{P}${\rm)}.
\end{proposition}

\begin{proof} Let $M$ be a Cartan factor of rank $\geq 3$. Let $e$ be a minimal tripotent in $M^{**},$ and let $u$ be a complete tripotent in $M$ {\rm(}that is $u\in \partial_e(\mathcal{B}_{M})${\rm)}. By \eqref{eq positive element FR} the element $P_2(u) \{e,e,u\}$ is positive in $M^{**}_2(u)$. On the other hand, by Peirce arithmetic, $\{e,e,u\} = P_2(e) (u) + \frac12 P_1(e)(u),$ where $P_2(e) (u) \in M_2^{**} (e) = \mathbb{C} e$, and hence $P_2(e) (u) = \delta e$ for some $\delta \in \mathbb{C}$. We observe that $u$ is also complete in $M^{**}$.\smallskip

By \cite[Corollary 2.2]{DanFri87} the subtriple $M^{**}_1(e)$ has rank at most two. Therefore, $\frac12 P_1(e)(u) = \lambda_1 v_1 +\lambda_2 v_2$, where $\lambda_1,\lambda_2\in \mathbb{R}_0^+$ and $v_1$ and $v_2$ are mutually orthogonal minimal tripotents in $M_1^{**} (e)$ or zero with $v_j \neq 0$ if $\lambda_j >0$.\smallskip

We shall distinguish several cases:\smallskip

\emph{Case 1:} $\lambda_1, \lambda_2>0$. By the triple system analyzer \cite[Proposition 2.1$(iii)$]{DanFri87}, $v_1,v_2$ are minimal tripotents in $M^{**}$ and the triplet $(v_1,e,v_2)$ is a prequadrangle in the terminology of \cite{DanFri87}. Therefore the three points $e,v_1,v_2$ are contained in a rank two JBW$^*$-subtriple of $M^{**}$. \smallskip

\emph{Case 2:} $\lambda_1>0$, $\lambda_2 =0$. By the triple system analyzer \cite[Proposition 2.1$(i)$ and $(ii)$]{DanFri87}, $v_1$ is a minimal tripotent in $M^{**}$ and $e,v_1$ are collinear; or $v_1$ is a minimal tripotent in $M_1^{**}(e)$ but not minimal in $M^{**}$ and there exists a minimal tripotent $\widetilde{e}$ in $M^{**}$ such that $(e,v_1,\widetilde{e})$ is a trangle in the terminology of \cite{DanFri87}. That is, the three points $e,v_1,v_2$ is contained in a rank one or two JBW$^*$-subtriple of $M^{**}$ (cf. \cite[LEMMA in page 306]{DanFri87}).\smallskip

\emph{Case 3:} $\lambda_1=\lambda_2 =0$, or equivalently, $P_1(e)(u) =0$. In this case $\{e,e,u\} =\delta e$ is contained in a rank one JBW$^*$-subtriple of $M^{**}$.\smallskip

We shall first assume that $e\in M_2^{**}(u) \cup M_1^{**}(u)$.\smallskip

It follows from the above cases that the set $\{P_2(e) (u) , \frac12 P_1(e)(u)\}$ is contained in a JBW$^*$-subtriple $F$ of $M^{**}$ of rank at most two. Since we have assumed that $e\in M_2^{**}(u) \cup M_1^{**}(u)$, the element $\{e,e,u\}$ coincides with $P_2(u) \{e,e,u\}$ and it is a positive element in the JBW$^*$-algebra $M_2^{**} (u)$. We also know that $\{e,e,u\}= P_2(e)(u) +\frac12 P_1(e)(u)\in F.$ The range tripotent of $\{e,e,u\}$ in $F$ and in $M_2^{**} (u)$ give the same element which will be denoted by $r$. Clearly, $r$ is a projection in $M_2^{**} (u)$ which must be minimal or the sum of two mutually orthogonal minimal projections in $M_2^{**} (u)$.\smallskip

Now, having in mind that $M$, and hence $M_2(u)$ and $M_2^{**} (u),$ all have rank $\geq 3$ (cf. \cite[Proposition 5.8]{Ka97}), we deduce the existence of a minimal projection $w$ in $M_2^{**} (u)$ which is orthogonal to $r$ (and hence to $\{e,e,u\}$). By applying that $r$ is the range tripotent of $\{e,e,u\}$, and the fact that $$M_2^{**}(u)= \left(M_2^{**}(u)\right)_2(r) \oplus \left(M_2^{**}(u)\right)_1(r) \oplus \left(M_2^{**}(u)\right)_0(r) ,$$ where $\left(M_2^{**}(u)\right)_0(r) = \left(M_2^{**}(u)\right)_2(u-r)$ and $\left(M_2^{**}(u)\right)_1(r) = \left(M_2^{**}(u)\right)_1(u-r),$ we deduce from the Jordan identity that
$$\{e,e,u-r\}+\{e,e,r\}=\{e,e,u\}= \{r,\{e,e,u\},r\} $$
$$= -\{e,e,\{r,u,r\}\} +2\{\{e,e,r\},u,r\}=- \{e,e,r\} + 2\{\{e,e,r\},r,r\}.$$
Therefore $$ \{e,e,u-r\} = - 2 \{e,e,r\} + 2\{\{e,e,r\},r,r\} \in \left(M_2^{**}(u)\right)_2(r) \oplus \left(M_2^{**}(u)\right)_1(r),$$
which implies that $0= P_2(u-r) \{e,e,u-r\} =  \{P_2(u-r) (e) , P_2(u-r)(e),u-r\}+ \{P_1(u-r)(e),P_1(u-r)(e),u-r\} $. Lemma 1.5 in \cite{FriRu85}, and the comments preceding it, now assure that $P_2(u-r) (e) = P_1(u-r) (e)=0$, and thus $e = P_0(u-r) (e) \perp u-r$. Since $w$ is a minimal projection in $M_2^{**} (u)$ with $w\leq u-r$, it follows that $w$ is a minimal tripotent in $M^{**}$ with $w\perp e$.\smallskip

We consider now the general case in which $e\in M_2^{**}(u) \oplus M_1^{**}(u)$. By Lemma \ref{l minimal respect complete}, $e$ decomposes as a linear combination of the form $e=\lambda v_2+\mu v_1$, where $v_2$ and $v_1$ are two collinear tripotents in $M^{**}$ which are  minimal or zero, $v_2\in M^{**}_2(u)$, $v_1\in M^{**}_1(u)$, and $\lambda,\mu\in \mathbb{R}^+_0$ satisfy $|\lambda|^2+|\mu|^2=1$.\smallskip

If $\lambda\neq 0$, we apply the first part of this proof to the element $v_2$ to find a minimal tripotent $w\in M^{**}$ such that $w\leq u$ and $w\perp v_2$. We shall next show that $w\perp v_1$. Indeed, the element $v_2 +w$ is a tripotent in $M^{**}_2 (u)$, and the corresponding Peirce projections commute, that is, $P_j(v_2) P_k(w)= P_k(w) P_j(v_2)$ for all $j,k\in \{0,1,2\}$ (cf. \cite[(1.10)]{Horn87}). The element $P_1(w) (v_1) =  P_1(w) P_1(v_2) (v_1) = P_1(v_2) P_1(w) (v_1)\in M_1^{**}(w) \cap M_1^{**}(v_2)\subset M^{**}_2(w+v_2)\subset M_2^{**}(u)$. However, $P_1(w) (v_1) =  P_1(w) P_1(u) (v_1)  =  P_1(u)  P_1(w) (v_1)\in M_1^{**} (u)$, and thus $P_1(w) (v_1) =0$. Moreover, $P_2(w) (v_1)\in M_2^{**} (w) \subset M_2^{**} (u),$ and $P_2(w) (v_1) = P_2(w) P_1(u) (v_1) = P_1(u)  P_2(w) (v_1)  \in M_1^{**} (u),$ which shows that $P_2(w) (v_1) =0 $, and consequently $v_1 = P_0(w) (v_1) \perp w$, as desired.\smallskip

Finally, if $\lambda =0$ we get $e = \mu v_1\in M^{**}_1(u)$ and we finish by applying the first part of this proof.
\end{proof}

We can now prove that most of JBW$^*$-triples satisfy the Mazur--Ulam property.

\begin{theorem}\label{t JBW-triples satisfy the MUP} Let $M$ be a JBW$^*$-triple with rank bigger than or equal to three. Then, every surjective isometry from the unit sphere of $M$ onto the unit sphere of a real Banach space $Y$ admits a unique extension to a surjective real linear isometry from $M$ onto $Y$.
\end{theorem}

\begin{proof} Let $\Delta : S(M) \to S(Y)$ be a surjective isometry from the unit sphere of $M$ onto the unit sphere of a Banach space $Y$. If we show that $M$ satisfies property {\rm(}$\mathcal{P}${\rm)}, then it follows from Corollary \ref{c technical before theorem 2} that, for each minimal tripotent $e$ in $M^{**}$ and each $\psi$ in $\hbox{supp}_{\Delta}(F_e^{M})$ we have $\psi \Delta (x) = \Re\hbox{e}\varphi_e(x)$ for every $x\in S(M)$, where $\varphi_e\in\partial_{e}(\mathcal{B}_{M^*})$ is the unique pure atom associated with a minimal tripotent $e$ in $M^{**}$. Let $\mathcal{U}_{min} (M^{**})$ denote the set of all minimal tripotents in $M^{**}$. For each $e\in \mathcal{U}_{min} (M^{**})$, we pick $\psi_e\in \hbox{supp}_{\Delta}(F_e^{M})$ (cf. Lemma \ref{l existence of support functionals for the image of a face}). We consider the families $\{\varphi_e\}_{e \in \mathcal{U}_{min} (M^{**})}$ and $\{\psi_e\}_{e \in \mathcal{U}_{min} (M^{**})}$. Since the set $\{\varphi_e\}_{e \in \mathcal{U}_{min} (M^{**})}$  is norming on $M$, and $\psi_e \Delta (x) =\Re\hbox{e} \varphi_e(x)$ for every $x\in S(M)$, the conclusion of the theorem will follow from \cite[Lemma 6]{MoriOza2018}.\smallskip

We shall finally prove that $M$ satisfies property {\rm(}$\mathcal{P}${\rm)}. Let $e$ be a minimal tripotent in $M^{**},$ and let $u$ be a complete tripotent in $M$ {\rm(}that is $u\in \partial_e(\mathcal{B}_{M})${\rm)}. By considering the atomic decomposition of $M^{**}$, we can write $M^{**} = \mathcal{A}\oplus \mathcal{N},$ as the $\ell_{\infty}$-direct (orthogonal) sum of its atomic and its non-atomic part.
The atomic part of $M^{**}$, $\mathcal{A},$ is precisely the weak$^*$-closure of the linear span of all minimal tripotents in $M^{**}$ (see \cite[Theorem 2]{FriRu85}). It is also known that $\mathcal{A} = \bigoplus_{j \in \Lambda} C_j$, where $\{C_j : j \in \Lambda\}$ is a family of Cartan factors (cf. \cite[Corollary 1.8]{Horn87b} and \cite[Proposition 2]{FriRu86}). It is further known that if $\imath_{M} : M \to M^{**}$ and $\pi_{at} : M^{**}\to \mathcal{A}$ denote the canonical inclusion of $M$ into its bidual and the projection of $M^{**}$ onto $\mathcal{A}$, respectively, then the mapping $\Phi = \pi_{at} \circ \imath_{M}$ is an isometric triple isomorphism with weak$^*$ dense image.\smallskip

The element $\Phi(u) = (u_j)_{j\in \Lambda}$ is a complete tripotent in $\mathcal{A}$, and $e$ belongs to a unique $C_{j_0}$. If $\sharp \Lambda\geq 2,$ we can find $j_1\neq j_0$ in $\Lambda$, and in this case, any minimal tripotent $w\in C_{j_1}$ with $w\leq u_{j_1}$ satisfies $w\leq u_{j_1}\leq \Phi(u) \leq u$ and $w\perp e$. We can therefore reduce to the case in which $\Lambda$ is a single element, and hence $\mathcal{A}$ is a Cartan factor. In the latter case the desired conclusion follows from Proposition \ref{p Cartan factors fo rank bigger than or equal to three satisfy P } because the rank of $M$ is smaller than or equal to the rank of $\mathcal{A}$.
\end{proof}

We shall finish this note by exploring the Mazur--Ulam property in the case of JBW$^*$-triples of rank one. Suppose $M$ is a JBW$^*$-triple of rank one. It is known that $M$ must be reflexive (see, for example, \cite[Proposition 4.5]{BuChu}). In particular $M$ must coincide with a rank one Cartan factor, and hence it must be isometrically isomorphic to a complex Hilbert space (cf. \cite[Table 1 in page 210]{Ka97}). It is due to G.G. Ding that every surjective isometry between the unit spheres of two Hilbert spaces admits a unique extension to a surjective real linear isometry between the spaces (see \cite{Ding2002}). Suppose now that $\Delta : S(H)\to S(Y)$ is a surjective isometry, where $H$ is a Hilbert space and $Y$ is a Banach space. Given $x,y\in S(Y)$ there exist $a,b\in S(H)$ satisfying $\Delta(a) =x $ and $\Delta(b)= y$. Having in mind that the set $\{b\}$ is maximal norm closed face of $\mathcal{B}_{H}$, we deduce from \cite[Lemma 8]{MoriOza2018} that $\Delta (-b) = -y$. Therefore, $$\left\| x+ y \right\|^2 + \left\| x- y \right\|^2 = \left\| \Delta(a)+ \Delta(b) \right\|^2 + \left\| \Delta(a)- \Delta(b) \right\|^2$$ $$ = \|a + b\|^2+ \|a - b\|^2 = 2 \|a\|^2 + 2\|b\|^2 = 4 .$$ It follows from \cite[Theorem 2.1]{Day47} that $Y$ is a Hilbert space. The previously quoted result of Ding in \cite{Ding2002} gives the next result.

\begin{proposition}\label{p rank one JBW-triples} Every Hilbert space satisfies the Mazur--Ulam property. Every rank one JBW$^*$-triple satisfies the Mazur--Ulam property. $\hfill\Box$
\end{proposition}

\begin{remark}\label{remark final}{\rm We have shown in the first part of the proof of Theorem \ref{t JBW-triples satisfy the MUP} that if $M$ is a JBW$^*$-triple satisfying property $(\mathcal{P})$, then, every surjective isometry from the unit sphere of $M$ onto the unit sphere of a Banach space $Y$ admits a unique extension to a surjective real linear isometry from $M$ onto $Y$. The proof of Theorem \ref{t JBW-triples satisfy the MUP} actually shows that every JBW$^*$-triple with rank bigger than or equal to three satisfies property $(\mathcal{P})$. We shall see next that there are other examples of JBW$^*$-triples satisfying property $(\mathcal{P})$.\smallskip

Suppose $M$ is a JBW$^*$-triple such that the atomic part of $M^{**}$ is not a Cartan factor of rank one or two (in particular when $M$ is not a factor). We claim that $M$ satisfies property $(\mathcal{P})$. Indeed, let $\mathcal{A}$ denote the atomic part of $M^{**}$. If $\mathcal{A}$ is a Cartan factor of rank bigger than or equal to three the proof of Theorem \ref{t JBW-triples satisfy the MUP} shows that $M$ satisfies property $(\mathcal{P})$. If $\mathcal{A}$ is an $\ell_{\infty}$-sum of at least two Cartan factors we have also seen in the proof of Theorem \ref{t JBW-triples satisfy the MUP} that $M$ satisfies property $(\mathcal{P})$.\smallskip

Let $E$ be a JB$^*$-triple. If the atomic part, $\mathcal{A}$, of $E^{**}$ is a Cartan factor of rank 2, or even more generally, a finite rank JBW$^*$-triple, then $\mathcal{A}$ is a reflexive Banach space (cf. \cite[Proposition 4.5]{BuChu} and \cite[Theorem 6]{ChuIo90}). As we have already commented, $E$ embeds isometrically into $\mathcal{A}$, and thus $E$ is reflexive and $E=E^{**}= \mathcal{A}$. \smallskip

In particular, if the atomic part of $E^{**}$ reduces to a rank one Cartan factor, then $E$ is a rank one JBW$^*$-triple and satisfies the Mazur--Ulam property by Proposition \ref{p rank one JBW-triples}.\smallskip

Summarizing, if $M$ is not a rank two Cartan factor, then $M$ satisfies the Mazur--Ulam property.}
\end{remark}

\medskip

\textbf{Acknowledgements} We would like to thank the anonymous referee for an outstanding and professional revision which helped us to avoid some difficulties appearing in earlier versions.\smallskip

Authors partially supported by the Spanish Ministry of Economy and Competitiveness (MINECO) and European Regional Development Fund project no. MTM2014-58984-P and Junta de Andaluc\'{\i}a grant FQM375.

\end{document}